\documentclass[11pt]{article}
\usepackage{amsfonts}
\usepackage{amssymb,latexsym}
\usepackage{amsmath,amscd}
\usepackage{theorem}
\usepackage{color}
\usepackage{combelow}
\usepackage[all]{xy}
\usepackage[pdftitle={Paper}, pdfauthor={Alejandro Cabrera}]{hyperref}
\usepackage{enumitem}
\usepackage{tensor}
\usepackage[affil-it]{authblk} 

\newtheorem{lemma}{Lemma}
\newtheorem{proposition}{Proposition}

\newtheorem{theorem}{Theorem}
\newtheorem{corollary}{Corollary}

\theorembodyfont{\rm}
\newtheorem{example}{Example}
\newtheorem{remark}{Remark}
\newenvironment{proof}{{\sc Proof:}}{{\hspace*{\fill} $\square$\\}}
\numberwithin{}{}

\newcommand{\X}{\ensuremath{\mathfrak{X}}}
\newcommand{\g}{\mathfrak{g}}

\newcommand{\tto}{\rightrightarrows}

\newcommand{\dto}{\dashrightarrow}
\newcommand{\rmap}{\to}
\newcommand{\diffto}{\xrightarrow{\raisebox{-0.2 em}[0pt][0pt]{\smash{\ensuremath{\sim}}}}}

\newcommand{\Lie}{\mathcal{L}}

\def\bea{\begin{eqnarray}}
\def\eea{\end{eqnarray}}
\def\bl{\begin{lemma}}
\def\el{\end{lemma}}
\def\br{\begin{remark}}
\def\er{\end{remark}}

\def\rr{\rightrightarrows}

\def\R{\mathbb{R}}
\def\C{C^\infty}

\def\s{\sigma}
\def\t{\tau}
\def\e{\epsilon}

\def\lef{\leftthreetimes}
\def\TT{\mathbb{T}}
\def\br{\bar{\rho}}
\newcommand{\Cour}[1]      {[\![#1]\!]}

\newcommand{\ale}[1]{{\tiny #1}\marginpar{*A} }

\def\GG{\mathbb{G}_k}
\def\AA{\mathbb{A}_k}
\def\MM{\mathbb{M}_k}

\begin{document}

\title{Local formulas for multiplicative forms}

\author{Alejandro Cabrera}
\affil{Departamento de Matem\'atica Aplicada, Instituto de Matem\'atica\\
Universidade Federal do Rio de Janeiro\\
 Caixa Postal 68530, Rio de Janeiro RJ 21941-909, Brasil\\ alejandro@matematica.ufrj.br}

\author{Ioan M\u{a}rcu\cb{t}}
\affil{Radboud University Nijmegen, IMAPP\\ 
6500 GL, Nijmegen, The Netherlands\\
i.marcut@math.ru.nl}
 
\author{Mar\'ia Amelia Salazar}
\affil{Instituto de Matem\'atica e Estat\'istica, GMA\\ 
Universidade Federal Fluminense\\ 
Rua Professor Marcos Waldemar de Freitas Reis s/n, Gragoat\'a, Niter\'oi, Rio de Janeiro, Brasil, 24.210-201\\ mariasalazar@id.uff.br}

\maketitle

\begin{abstract}
We provide explicit formulas for integrating multiplicative forms on local Lie groupoids in terms of infinitesimal data. Combined with our previous work \cite{love}, which constructs the local Lie groupoid of a Lie algebroid, these formulas produce concrete integrations of several geometric stuctures defined infinitesimally. In particular, we obtain local integrations and non-degenerate realizations of Poisson, Nijenhuis-Poisson, Dirac, and Jacobi structures by local symplectic, symplectic-Nijenhuis, presymplectic, and contact groupoids, respectively. 


%
\end{abstract}

\tableofcontents

\section{Introduction}\label{sec:intro}

In recent years, the study of Poisson, symplectic and related geometries has been connected at a fundamental level to the underlying Lie theory of algebroids and groupoids \cite{CDW,CF1,BCWZ,CrainicZhu,BaGu}. In these contexts, Lie groupoids come endowed with additional geometric structures which are compatible with the groupoid structure, called \textbf{multiplicative}. A paradigmatic example of this interplay is the correspondence between Poisson manifolds and \textbf{symplectic groupoids} (\cite{CDW}):  Lie groupoids endowed with a multiplicative symplectic structure.


In general, a \textbf{multiplicative form} on a Lie groupoid is a differential form satisfying a certain cocycle-type condition. Multiplicative forms can be differentiated to the so-called \textbf{infinitesimally multiplicative} (or \textbf{IM}-) forms on the corresponding Lie algebroid, or more generally to {\bf Spencer operators} when the multiplicative form has coefficients in a representation. For source one-connected Lie groupoids, the differentiation procedure gives a one-to-one correspondence \cite{CamiloMarius,BC,Maria}. In this paper we prove the local version of this result, namely, differentiation gives a one-to-one correspondence between germs of multiplicative forms on a local Lie groupoid and IM-forms/Spencer operator on the corresponding Lie algebroid. Moreover, our main result consists in providing an explicit formula for an integration in terms of the IM-data, which depends on a chosen tubular neighborhood of the identity section of the groupoid (Theorem \ref{thm:tubular_Spencer}). In our previous work \cite{love}, we presented an explicit construction of a local Lie groupoid integrating a given Lie algebroid, called the spray groupoid, and in this case, the integration of IM-forms and Spencer operators procedure simplifies substantially (Theorem \ref{theorem:Spencer}).

We also apply these integration results in the context of  Poisson, Nijenhuis-Poisson, Dirac and Jacobi geometries, as mentioned above. The key point is that each of these geometries can be encoded by a Lie algebroid together with an IM-form. At the groupoid level, these correspond to geometric structures of a similar nature, which are multiplicative differential forms satisfying a specific non-degeneracy condition. 
By using the spray construction \cite{love}, and the integration procedure, we construct explicit local symplectic, symplectic-Nijenhuis, presymplectic, and contact groupoids; and since the source map of these local Lie groupoids is a morphism in the same category, we obtain concrete non-degenerate realizations for several types of structures: Poisson, Nijenhuis-Poisson (including holomorphic Poisson), generalized complex, twisted Dirac, and Jacobi. This yields new proofs of the main results in the recent works \cite{BX,CM,IP,Peta}.

\vspace*{0.3cm}

We explain now our main result in the case of the spray groupoid \cite{love} and for trivial coefficients (Theorem \ref{theorem:spray_no_coefficeints}). First, we recall some terminology and results from \cite{BC}. Let $(A,[\cdot,\cdot],\rho)$ be a Lie algebroid over a manifold $M$. An \textbf{IM-$k$-form} on $A$ consists of a pair of vector bundle maps
\[(l,\nu):A\rmap \wedge^{k-1}T^*M\oplus \wedge^kT^*M\]
satisfying, for all $a,b\in \Gamma(A)$, the compatibility relations
\begin{align}
\nonumber&\nu([a,b])=\mathcal{L}_{\rho(a)}\nu(b)-i_{\rho(b)}d\nu(a)\\
&l([a,b])=\mathcal{L}_{\rho(a)}l(b)-i_{\rho(b)}dl(a)-i_{\rho(b)}\nu(a)\label{eq:IM}\\
\nonumber&i_{\rho(a)}l(b)=-i_{\rho(b)}l(a).
\end{align}
Let $G\rr M$ be a Lie groupoid integrating $A$. The structure maps are denoted as follows: source and target maps $\sigma,\tau:G\to M$, unit map $u:M\to G$, inversion $\iota:G\to G$, and multiplication map $\mu:G_{\sigma}\times_{\tau}G\to G$. A \textbf{multiplicative $k$-form} on $G$ is a differential form $\omega\in \Omega^k(G)$ satisfying the compatibility relation
\begin{equation}\label{eq:M}
\mu^*\omega=\mathrm{pr}_1^*\omega+\mathrm{pr}_2^*\omega
\end{equation}
on the space of composable arrows $G_{\sigma}\times_{\tau}G$. Multiplicative $k$-forms can be differentiated to IM-$k$-forms: the \textbf{differentiation} of $\omega$ is the pair $(l,\nu)$ given by
\begin{equation}\label{eq:differentiation}
l(a):=u^*(i_{a}\omega), \ \ \ \nu(a):=u^*(i_{a}d\omega), \ \ \forall  \ a\in \Gamma(A).
\end{equation}
If the source fibers of $G$ are one-connected, then the differentiation map is a bijection between multiplicative forms on $G$ and $IM$-forms on $A$ \cite{CamiloMarius,BC}.

Not every Lie algebroid is integrable by a Lie groupoid, but every Lie algebroid is integrable by a local Lie groupoid \cite{CF1} (regarding local Lie groupoids, we use the conventions from \cite{love}). If $G$ is only a local Lie groupoid integrating $A$, one can still talk about local multiplicative forms, i.e.\ differential forms defined around the unit section of $G$ which satisfy (\ref{eq:M}) around the unit section, and  moreover, the differentiation operation is still defined and produces IM-forms on $A$. In this setting, our main result shows that differentiation gives a bijection from germs around the unit section of local multiplicative forms on $G$ to IM-forms on $A$, with the aid of an explicit formula. We explain this for spray groupoids. Recall from \cite{love}, that a \textbf{Lie algebroid spray} for $A$ is a one-homogenous vector field $V\in \mathfrak{X}(A)$ satisfying $dq(V_a)=\rho(a)$, for all $a\in A$, where $q:A\to M$ denotes the bundle projection. Using a spray $V$, in \cite{love} we build a local Lie groupoid $G_V\rr M$ integrating $A$, with total space an open neighborhood $G_V\subset A$ of the zero-section, with $u(m)=0_m$, with $\sigma=q|_{G_V}$, with $\tau=q\circ \phi_V^1$, and with $\iota=-\phi_V^{1}$, where $\phi_V^t$ denotes the local flow of $V$. To define the multiplication $\mu$, we first described the Maurer-Cartan form of $G_V$ which allows to write $\mu(a,b)$ as the solution to an ODE. We state now our main result in this particular case:
\begin{theorem}\label{theorem:spray_no_coefficeints}
Differentiation is a bijection between germs of locally defined multiplicative forms on the spray groupoid $G_V$ and IM-forms on its Lie algebroid $A$, with inverse $(l,\nu)\mapsto \omega$ is given by the integral formula:
\[\omega=\int_0^1(\phi_V^t)^*\Lambda_{(l,\nu)}\ dt,\]
where $\Lambda_{(l,\nu)}:=l^*(d\alpha^{k-1}_{\mathrm{can}})+ \nu^*(\alpha^{k}_{\mathrm{can}})$, and $\alpha^{p}_{\mathrm{can}}\in \Omega^p(\wedge^pT^*M)$ denotes the tautological $p$-form. 
\end{theorem}

Finally, let us illustrate our result for the problem of locally integrating a Poisson manifold $(M,\pi)$. Using an ordinary connection $\nabla$ on $M$, we can construct the spray $V$ on $T^*M$ defined at $a\in T^*M$ as the horizontal lift of $\pi^{\sharp}(a)$ with respect to $\nabla$. The resulting \textbf{local symplectic groupoid} constructed here, which integrates the cotangent Lie algebroid $T^*_{\pi}M$, is given by
\[(G_V,\omega)\rr (M,\pi), \ \ \ \omega=\int_0^1(\phi_V^t)^*\omega_{0}\ dt,\]
where $G_V\subset T^*M$ is a neighborhood of $M$, and $\omega_{0}$ is the canonical symplectic structure on $T^*M$; thus $\omega$ is the multiplicative 2-form integrating the IM-2 form $(-\mathrm{id},0)$. Since the Maurer-Cartan form of $G_V$ is induced by the symplectic structure, the multiplication can be described more explicitely: $\mu(a,b)=k_1$, where $k_t\in T^*_{q(b)}M$, $t\in [0,1]$, is the solution of the ODE:
\[i_{\frac{dk_t}{dt}}\omega=-\tau^*(\phi_V^t(a)),\ \ k_0=b.\]



\textbf{Acknowledgments.} The authors would like to thank Marius Crainic, Pedro Frejlich, Rui Loja Fernandes and Eckhard Meinrenken for useful discussions. I.M.\ was supported by the NWO Veni grant 613.009.031 and the NSF grant DMS 14-05671. M.A.S.\ was a Post-Doctorate at IMPA, funded by CAPES-Brazil, during part of this project, and was partly supported by the CNPq Universal grant 409552/2016-0. A.C. would like to thank CNPq and FAPERJ for financial support.

\section{Differentiation and integration of multiplicative forms}

\subsection{Local multiplicative forms with coefficients}

Let $G\rr M$ be a \textbf{local Lie groupoid}. We follow the conventions from \cite{love} about local Lie groupoids. In order to ease notation, we use a dashed arrow $\dto$ to indicate that the domain of a map is an open neighborhood of the unit section $u:M\to G$; for example, the structure maps are denoted by 
\[\textrm{source and target}\  \sigma, \tau:G\dto M, \ \  \textrm{inversion}\  \iota: G\dto G,  \ \ \textrm{multiplication}\  \mu:G_{\sigma}\times_{\tau}G\dto G,\]
and the axioms of a Lie groupoid are satisfied locally. 

The \textbf{Lie algebroid} $A\to M$ of $G$ is constructed exactly as in the case of ordinary Lie groupoids: as a vector bundle $A=\ker(d\sigma)|_M$, the anchor of $A$ is $\rho=d\tau|_A:A\to TM$, and the Lie bracket is obtained by identifying sections of $A$ with (germs of) right invariant vector fields on $G$.

The \textbf{Maurer-Cartan} form $\theta_G$ of $G$ refers to the vector bundle map covering the target map, given by right translation of vector field tangent to the source fibers to the Lie algebroid: 
\[\theta_G:T^\sigma G\to A, \ \ \theta_G\Big(\frac{d }{dt}g_t\big|_{t=0}\Big):=\frac{d}{dt}g_tg_0^{-1}\big|_{t=0}\in A_{\tau(g_0)},\]
where $T^{\sigma}G:=\ker (d\sigma)$, and $t\mapsto g_t$ is any path so that $\sigma(g_t)=x$ is fixed.

The \textbf{general linear groupoid} $GL(E)\rightrightarrows M$ of a vector bundle $p:E \to M$ consists of linear isomorphism $E_x\diffto E_y$, $x,y\in M$. A representation of a local Lie groupoid $G\rightrightarrows M$ on $E$ is defined as a local Lie groupoid map $F:G\dto GL(E)$ covering the identity map on $M$. Thus, for any arrow $g\in G$ (which is close enough to $M$), one has an associated linear isomorphism
\[F(g):E_{\sigma(g)} \diffto E_{\tau(g)}, \  e \mapsto g\cdot e,\]
satisfying $(gh)\cdot e=g\cdot(h\cdot e)$, for all composable arrows $(g,h)$ close enough to the identity section. 

Given a representation $E$ of $G$, a \textbf{local $E$-valued $k$-form} on $G$ is an ordinary differential $k$-form on an open neighborhood $U$ of $M$ in $G$ with values in the vector bundle $\sigma^*(E)$:
\[\omega \in \Omega^k(U; \sigma^*(E)|_U).\]
We say that $\omega$ is \textbf{multiplicative} if it satisfies the following equation
\begin{equation}\label{eq:multwE}
\mu^* (\omega_{gh})=h^{-1} \cdot\mathrm{pr}_1^* (\omega_g) +\mathrm{pr}_2^*(\omega_h),
\end{equation}
for $(g,h)$ in a neighborhood of $M$ in $G_{\sigma}\times_{\tau}G$. For the trivial representation on $E=\R\times M$, we recover the notion of multiplicative forms (\ref{eq:M}) without coefficients from the introduction. 

\subsection{Spencer operators}

A \textbf{representation} of a Lie algebroid $A$ over $M$ on the vector bundle $E$ is a flat $A$-connection $\nabla$ on $E$, i.e.\ a bilinear map 
\[\nabla:\Gamma(A)\times\Gamma(E)\rmap\Gamma(E), \ \ (a,e)\mapsto \nabla_ae\]
satisfying
\[\nabla_{fa}e=f\nabla_ae,\ \ \ \nabla_{a}fe=f\nabla_ae+\mathcal{L}_{\rho(a)}(f)e, \ \ \ \nabla_a\nabla_b-\nabla_b\nabla_a=\nabla_{[a,b]},\]
for all $f\in C^{\infty}(M)$, $a,b\in \Gamma(A)$ and $e\in \Gamma(E)$.

The infinitesimal version of $E$-valued multiplicative $k$-forms are \textbf{Spencer operators} of degree $k$ on a Lie algebroid $A$, with values in a representation $(E,\nabla)$ of $A$ \cite{Maria, Mtesis}. These consist of a vector bundle map $l$ and a linear operator $D$:
\[l:A\rmap \wedge^{k-1}T^*M\otimes E\ \ \ \ \ \ D:\Gamma(A)\rmap \Omega^k(M;E)\]
satisfying the Leibniz identity
\begin{equation}\label{eq:Leibniz}D(fa)=fD(a)+df\wedge l(a),\end{equation}
and the three compatibility conditions:
\begin{equation}\label{eq:compatibility}
\begin{split}
  D([a,b] ) = \Lie_{a}D(b) - \Lie_{b}D(a),\ \
  l([a,b]) = \Lie_{a}l(b) - i_{\rho(b)}D(a), \ \
  i_{\rho(a)}l(b) = -i_{\rho(b)}l(a),
 \end{split}
 \end{equation}
for all $a,b\in \Gamma(A)$, and $f\in \C(M)$; the operator $\Lie_{a}$ acts on $\Omega^k(M, E)$ by
\begin{equation}\label{eq:Lie}
\Lie_{a}\omega(v_1, \ldots, v_k) = \nabla_{a}(\omega(v_1, \ldots, v_k)) - \sum_{i=1}^k\omega(v_1, \ldots, [\rho(a),v_i], \ldots , v_k),\ \ \forall\ v_i\in \mathfrak{X}(M).
\end{equation}

In order to describe integration and differentiation operations, we introduce the \textbf{linear differential form} associated to a Spencer operators. Let $q_1,q_2:F_1,F_2\to N$ be two vector bundles. By a \textbf{linear $F_2$-valued $k$-form} on the manifold $F_1$ we mean a differential form $\Lambda\in \Omega^{k}(F_1;q_1^*(F_2))$ satisfying 
\begin{equation}
\label{eq:linear}m_t^*(\Lambda)=t\Lambda, \ \ \forall \ t\in \R,
\end{equation}
where $m_t:F_1\to F_1$ denotes the multiplication by $t\in \R$. The space of all linear $F_2$-valued differential forms will be denoted by
$\Omega^\bullet_{\mathrm{lin}}(F_1;F_2)\subset \Omega^\bullet(F_1;q_1^*F_2)$. 

To any $F_2$-valued linear differential $k$-form $\Lambda$ we associate a
pair $(l,D)$ consisting of a vector bundle map $l:F_1\to \wedge^{k-1}T^*N\otimes F_2$ and an operator $D:\Gamma(F_1)\to \Omega^k(N;F_2)$ via the formulas
\begin{equation}\label{eq:linear_Spencer}
l(a)=j^*(i_{\overrightarrow{a}}\Lambda), \ \ \ D(a)=a^*(\Lambda), \ \ \forall a\in \Gamma(F_1),
\end{equation}
where, $\overrightarrow{a}$ denotes $a$ regarded as a constant vertical vector field on $F_1$ and $j:N\to F_1$ denotes the zero-section. The pair $(l,D)$ satisfies the Leibniz identity (\ref{eq:Leibniz}); in fact, the following holds:

\begin{lemma} \label{lem:one_to_one}
The relation \eqref{eq:linear_Spencer} gives a one-to-one correspondence between:
\begin{itemize}
		\item pairs $(l,D)$ consisting of a vector bundle map $l:F_1\to \wedge^{k-1}T^*N\otimes F_2$ and an operator $D:\Gamma(F_2)\to \Omega^k(N;F_2)$ satisfying the Leibniz identity (\ref{eq:Leibniz});
		\item elements $\Lambda\in \Omega^k_{\mathrm{lin}}(F_1;F_2)$, i.e.\ $F_2$-valued linear differential forms on the manifold $F_1$.
\end{itemize}
\end{lemma}
\begin{proof} We only sketch the proof by writing the correspondence in local coordinates. Let $x^i$ denote local coordinates on $U\subset N$, $f_j$ a frame of $F_1|_{U}$ with induced dual coordinates $y^j$ on the fibers of $F_1|_U$, and $r_l$ a local frame on $F_2|_U$. The correspondence is such that, if
	\[D(f_j)=u_{I_k,j}^{l}(x)dx^{I_k}\otimes r_l, \ \  l(f_j)=v_{I_{k-1},j}^{l}(x)dx^{I_{k-1}}\otimes r_l,\]
	where $I_p$ denotes an index set $I_p=\{i_1<\ldots<i_p\}$ and $dx^{I_p}:=dx^{i_1}\wedge\ldots \wedge dx^{i_p}$, then
	\[\Lambda=y^ju_{I_k,j}^{l}(x)dx^{I_k}\otimes r_l+v_{I_{k-1},j}^{l}(x)dy^j\wedge dx^{I_{k-1}}\otimes r_l.
	\]
	\end{proof}
	
For a Spencer operator $(l,D)$ of degree $k$ on a Lie algebroid $A$ with values in a representation $E$, we denote the corresponding \textbf{linear $E$-valued differential form} by 
\[\Lambda_{(l,D)}\in \Omega^k_{\mathrm{lin}}(A;E).\]

Let $E=\mathbb{R}\times M$ be the \textbf{trivial representation} (i.e.\ given by the anchor). Then the Leibniz identity (\ref{eq:Leibniz}) for a degree $k$ Spencer operator $(l,D)$ is equivalent to the fact that $D$ decomposes
\[D(a)=dl(a)+\nu(a),\ \ a\in \Gamma(A),\]
where $\nu:A\to \wedge^{k}T^*M$ is a vector bundle map. Writing the compatibility conditions \eqref{eq:compatibility} in terms of $(l,\nu)$, we obtain the relations (\ref{eq:IM}). Thus, Spencer operators on the trivial representation are the same as IM-forms on a Lie algebroid $A$, which were discussed in the introduction. The associated linear differential form will be denoted by
\[\Lambda_{(l,\nu)}:=\Lambda_{(l,D)}\in \Omega_{\rm{lin}}(A)\]
and can be written directly in terms of $(l,\nu)$ as follows:
\begin{equation}\label{eq:lin_form_IM}
\Lambda_{(l,\nu)}=dl^*(\alpha^{k-1}_{\mathrm{can}})+\nu^*(\alpha^k_{\mathrm{can}})\in \Omega^k_{\mathrm{lin}}(A),
\end{equation}
where $\alpha^{p}_{\mathrm{can}}\in \Omega^p(\wedge^pT^*M)$ denotes the tautological $p$-form on $\wedge^pT^*M$.

\subsection{Differentiation}

The \textbf{differentiation procedure} \cite{Mtesis}, which associates to a local $E$-valued multiplicative $k$-form $\omega$ on local Lie groupoid $G\rr M$ an $E$-valued Spencer operator $(l,D)$ of degree $k$ on its Lie algebroid $A$, is simply given by the relations
\begin{equation}\label{eq:Dfromomega}
l(a)=u^*(i_{a}\omega),\ \ \ D(a)=\frac{d}{dt}\big|_{t=0}u^*((\phi_{a^R}^t)^*\omega)
\end{equation}
for any $a\in\Gamma(A)$, where $a^R\in \Gamma(T^{\sigma}G)$ denotes the local right-invariant vector field on $G$ corresponding to $a$ which, using the Maurer-Cartan form, is given by $\theta_G(a^R)=a$. 
If $(l,D)$ is given by (\ref{eq:Dfromomega}), we say that $\omega$ \textbf{integrates} the Spencer operator $(l,D)$.

In the case of trivial coefficients $E=\mathbb{R}\times M$, using Cartan's formula $\mathcal{L}_{a^R}=di_{a^R}+i_{a^R}d$ we have that $D(a)=u^*(i_ad\omega)+dl(a)$, thus under the correspondence between Spencer operators and IM-forms, we recover the differentiation (\ref{eq:differentiation}) from the Introduction.

The linear form corresponding to the Spencer operator can be thought of as the linearization of the multiplicative form. To make this precise, we recall some terminology from \cite{love}.  A \textbf{tubular structure} on a local Lie groupoid $G\rr M$ with Lie algebroid $A$ is a (locally defined) tubular neighborhood of $M$ in $G$ along the source fibers, i.e.\ an open embedding of bundles
\begin{equation}\label{diagram:tubular}
\xymatrixrowsep{0.4cm}
\xymatrixcolsep{1.2cm}
\xymatrix{
 A \ar@{-->}[r]^{\varphi}\ar[d]_{q} & G\ar[d]^{\sigma} \\
 M\ar[r]^{\mathrm{id}_M} & M},\end{equation}
such that: $\varphi(0_x)=x$, for $x\in M$, and
$\frac{d}{dt}\varphi(ta)\big|_{t=0}=(0,a)\in T_xM\oplus A_x=T_{x}G,$
for $a\in A_x$.

\begin{lemma}\label{lemma:linearization_is_linearization}
Let $\omega$ be a local $E$-valued multiplicative $k$-form on the local Lie groupoid $G\rr M$, and let $(l,D)$ be the Spencer operator on the Lie algebroid $A$ of $G$ to which $\omega$ differentiates. For any tubular structure $\varphi:A\dto G$, the linear form associated with $(l,D)$ is given by:
\[\Lambda_{(l,D)}=\frac{d}{d t}m_{t}^*(\varphi^*(\omega))\big|_{t=0}\in \Omega_{\mathrm{lin}}^k(A;E),\]
where $m_{t}:A\to A$ denotes multiplication by $t\in \R$.
\end{lemma}
\begin{proof}
Let $\Lambda:=\frac{d}{d t}m_{t}^*(\varphi^*(\omega))\big|_{t=0}$. First, note that $\Lambda$ is indeed a linear form:
\[m_s^*(\Lambda)=m_s^*\frac{d}{d t}m_{t}^*(\varphi^*(\omega))\big|_{t=0}=
\frac{d}{d t}m_{st}^*(\varphi^*(\omega))\big|_{t=0}=s\Lambda.\]
By the correspondence (\ref{eq:linear_Spencer}) from Lemma \ref{lem:one_to_one}, and the description of differentiation (\ref{eq:Dfromomega}), it suffices to check the following two relations:
\[j^*(i_{\overrightarrow{a}}\Lambda)=u^*(i_a\omega), \ \ \ a^*(\Lambda)=u^*(\mathcal{L}_{a^R}\omega), \ \ \forall\ a\in \Gamma(A).\]
Since $(m_t)_{*}(\overrightarrow{a})=t\overrightarrow{a}$, $m_t\circ j=j$, $d\varphi(\overrightarrow{a}|_{j(M)})=a|_{u(M)}$, and $\varphi\circ j=u$, we have that 
\[j^*(i_{\overrightarrow{a}}m_t^*(\varphi^*(\omega)))=j^*(m_t^*(t i_{\overrightarrow{a}}\varphi^*(\omega)))=j^*(t i_{\overrightarrow{a}}\varphi^*(\omega))=t j^*(\varphi^*(i_{a}\omega))=t u^*(i_a\omega),\]
which implies the first relation. Denote $\alpha_1^t=\phi^t_{a^R}\circ u$ and $\alpha_2^t=\varphi(ta)$.
Note that the families $\alpha^t_i:M\to G$ are sections of $\sigma$, i.e.\ $\sigma\circ \alpha_i^t=\mathrm{id}_M$, and they both satisfy $\alpha^0_i=u$ and $\frac{d}{dt}\alpha_i^t|_{t=0}=a$. It is easy to check (e.g.\ by a local calculation) that for any two such families of maps, we have 
\[\frac{d}{dt}\big|_{t=0}(\alpha_1^t)^*\omega=\frac{d}{dt}\big|_{t=0}(\alpha_2^t)^*\omega.\]
The left hand side of this equation is $\frac{d}{dt}|_{t=0}u^*(\phi^t_{a^R})^*\omega$, and since $\varphi\circ m_t\circ a=\alpha_2^t$, the right hand side is $a^*(\Lambda)$. This proves the second relation.
\end{proof}

\subsection{Integration}

We are ready to state the main result of this paper, which provides a formula for the inverse of the differentiation of multiplicative forms on a local Lie groupoid. Let $G\rr M$ be a local Lie groupoid, let $A$ be the Lie algebroid of $G$ and $\theta_G$ be the Maurer-Cartan form of $G$. Fix also a tubular structure $\varphi:A\dto G$ on $G$. The tubular structure induces a scalar multiplication on $G$ along the source fibers
\[\R\times G\dto G,\ \ (t,g)\mapsto tg:=\varphi(t\varphi^{-1}(g)),\]
which is defined on a neighborhood of $\R\times u(M)$ in $\R \times G$. Next, we consider the map
\begin{equation}\label{eq:lambda_t}
\lambda:\R\times G\dto A,\ \ (t,g)\mapsto \lambda_t(g):=\theta_G\big(\frac{d}{dt} tg\big).
\end{equation}
Note that $\lambda_t(g)\in A_{\tau(tg)}$. The following will be proven in the following subsections:

\begin{theorem}\label{thm:tubular_Spencer} Given a Spencer operator $(l,D)$ of degree $k$ on $A$ with coefficients in $E$, with corresponding linear form $\Lambda=\Lambda_{(l,D)}$, the formula
\begin{equation}\label{eq:mult-form-coeff}
g\mapsto \omega_g=\int_{0}^1 (tg)^{-1}\cdot \lambda_t^*\big(\Lambda_{\lambda_t(g)}\big)\ dt \in  \wedge^kT^*_gG\otimes E_{\sigma(g)}
\end{equation}
defines a local $E$-valued multiplicative $k$-form $\omega$ on $G$ which integrates $(l,D)$. Moreover, the germ around $M$ of such an integration is unique.
\end{theorem}

Note the following consequence of the theorem (see also \cite[Proposition 8]{L-BM}):
\begin{corollary}
Let $G\rightrightarrows M$ be a local Lie groupoid with Lie algebroid $A$, and $E$ a local representation of $G$. Differentiation gives a one-to-one correspondence between germs around $M$ of $E$-valued multiplicative forms and $E$-valued Spencer operators.
\end{corollary}

\subsection{The case of cocycles}\label{subsection:cocyles}

We discuss now Theorem \ref{thm:tubular_Spencer} in the case of cocycles and trivial coefficients. A \textbf{local cocycle} on a local Lie groupoid $G\rr M$ is a degree zero multiplicative form on $G$, i.e. a smooth map $f:G\dto \R$ satisfying
\[f(gh)=f(g)+f(h),\]
for all arrows $g$ and $h$ closed enough to the identities. This is equivalent to $G$ being a Lie groupoid homomorphism. A \textbf{cocycle} on the Lie algebroid $A$ is a degree zero IM-form on $A$, which is given by a linear map $\delta:A\to \R$ satisfying
\[\delta([a,b])=\mathcal{L}_{\rho(a)}\delta(b)-\mathcal{L}_{\rho(b)}\delta(a),\]
for all $a,b\in \Gamma(A)$. This is equivalent to $\delta$ being a Lie algebroid homomorphism. 
In this case, differentiation becomes $\delta(a)=i_a(df)$, and $\Lambda_{(0,\delta)}=\delta\in \Gamma (A^*)$ is the underlying (1-)cocycle.

By \cite[Theorem 2.4]{love}, Lie algebroid maps can always be integrated to local Lie groupoid maps, and the germ of the integration around the unit section is unique. Moreover, \cite[Theorem 2.4]{love} gives the following recipe to construct the integration of morphisms. Namely, if $f:G\dto \R$ integrates $\delta:A\to \R$, then for $g\in G$, we have that $f(tg)=x_t$ is determined by the ODE
\[x_0=0, \ \ \frac{d}{dt}x_t=\delta\Big(\theta_G\Big(\frac{d}{dt}tg\Big)\Big),\]
which has as solution 
\[x_t=\int_{0}^t\delta(\lambda_s(g))\ ds.\]
Thus $f(g)=\int_{0}^1\delta(\lambda_t(g))dt$, and so this proves Theorem \ref{thm:tubular_Spencer} in degree zero, for trivial coefficients.

\subsection{Forms as cocycles on $\GG$}
Following \cite{BC,CDr,DrE}, we will prove Theorem \ref{thm:tubular_Spencer} by embedding degree $k$ multiplicative forms and Spencer operators as cocycles on suitable larger groupoids $\GG$ and algebroids $\AA$, respectively, which we now introduce.

The first ingredient is the tangent lift construction. Given a local Lie groupoid $G\rightrightarrows M$ , the application of the tangent functor to its structure maps defines a local Lie groupoid
\[TG \rightrightarrows TM\]
called the local \textbf{tangent groupoid} associated to $G$. Note that $TG$ is more than a local Lie groupoid, because its structure maps are vector bundle maps, covering the structure maps of $G$, and their domains are vector bundles over the domains of the structure maps of $G$; in fact, such an object will be called a \textbf{local VB-groupoid}, following the terminology of \cite{Raj}, see also \cite{Mac}. The infinitesimal counterpart of this construction is the \textbf{tangent Lie algebroid} structure on $dq:TA \to TM$ associated to a Lie algebroid $q:A \to M$ (see \cite{Mac} for further information).

If $A$ is the Lie algebroid of $G$, i.e.\ $A=T^{\sigma}G|_M$, then the Lie algebroid of $TG\rr TM$ is naturally isomorphic to $TA\to TM$ (see \cite{Mac}). The identification between the vector bundles $T^{d\sigma}(TG)|_{TM}\to TM$ and $TA=T(T^{\sigma}G|_M)\to TM$ is realized by the natural involution of $T(TG)$ which switches second order derivatives:
\begin{equation}\label{eq:can_involution}
\psi:T^{d\sigma}\big(TG\big)|_{TM}\diffto T\big(T^{\sigma}G|_{M}\big), \ \ \ \ \frac{d}{d t}\frac{d}{d \e} g_{\e} (t)\mapsto\frac{d}{d \e}\frac{d}{d t} g_{\e}(t),
\end{equation}
where $g_{\e}(t)\in G$ is a smooth two-parameter family, such that $\sigma(g_{\e}(t))=x_{\e}=g_{\e}(0)$. 
For later use, we record the following characterization of the Maurer-Cartan form on $TG$.

\begin{lemma}\label{lema:MC_TG}
The Maurer-Cartan form of $TG$ is given by:
\[\theta_{TG}\Big(\frac{d}{d t}\frac{d}{d \e} g_{\e} (t)\Big)=\frac{d}{d \e}\theta_G\Big(\frac{d}{d t} g_{\e} (t)\Big),\]
where $g_{\e}(t)\in G$ is a smooth two-parameter family, such that $\sigma(g_{\e}(t))=x_{\e}$ is independent of $t$.
\end{lemma}
\begin{proof}
We have that
\begin{align*}
\theta_{TG}\Big(\frac{d}{d t}\frac{d}{d \e} g_{\e} (t)\Big)&=\psi\circ \frac{d}{d s}\big|_{s=t}d\mu\Big(\frac{d}{d \e}g_{\e}(s),d\iota\big(\frac{d}{d \e}g_{\e}(t)\big)\Big)=
\psi\circ \frac{d}{d s}\big|_{s=t}\frac{d}{d \e}\mu\big(g_{\e}(s),\iota(g_{\e}(t))\big)=\\
&=\frac{d}{d \e}\frac{d}{d s}\big|_{s=t}\mu\big(g_{\e}(s),\iota(g_{\e}(t))\big)=\frac{d}{d \e}\theta_G\Big(\frac{d}{d t} g_{\e} (t)\Big).
\end{align*}
\end{proof}

\smallskip

The second ingredient is given by the action construction on a given representation. Let $G\rr M$ be a local Lie groupoid, and let $E$ be a representation of $G$. The dual representation of $E$ is defined by the action of $G$ on $E^*$ given by $(g\cdot \xi)(e):=\xi(g^{-1}\cdot v)$, for $\xi\in E^*_{\sigma(g)}$, $v\in E_{\tau(g)}$. The associated {\bf local action groupoid} $G\lef E^*=G\times_{M}E^* \rightrightarrows E^*$ is defined by the structure maps:
\[\widetilde{\sigma}(g,\xi)=\xi, \ \widetilde{\tau}(g,\xi)=g\cdot\xi,\ \widetilde{u}(\xi)=(u(p(\xi)),\xi), \ \widetilde{\mu}((g,h\cdot \xi),(h,\xi))=(gh,\xi), \ (g,\xi)^{-1}=(g^{-1},g\cdot\xi).\]
Infinitesimally, let $A$ be the Lie algebroid of $G$ and $\nabla$ the $A$-connection on $E$ differentiating the representation. The Lie algebroid of $G\lef E^*$ is the \textbf{action Lie algebroid} $A\lef E^*=A\times_{M}E^* \to E^*$; its structure maps are determined by two conditions: the canonical inclusion $\Gamma(A)\subset \Gamma(A\lef E^*)$ preserves the Lie bracket, and the anchor $\widetilde{\rho}$ satisfies
\begin{equation}\label{eq:rho_tilde}
\mathcal{L}_{\widetilde{\rho}(a)}\widetilde{e}=\widetilde{\nabla_ae}, \ \ \forall\ a\in \Gamma(A), \ e\in \Gamma(E),
\end{equation}
where for $f\in \Gamma(E)$, we denote by $\widetilde{f}\in C^{\infty}(E^*)$ the corresponding linear function on $E^*$.
Notice that $G\lef E^*$ is a local VB-groupoid and that $A\lef E^*$ is its VB-algebroid.

By taking sums of the underlying local VB-groupoid structure maps (see e.g.\ \cite{BC}), we get a local (VB-)groupoid structure on 
\[\GG := \overbrace{TG \oplus \cdots \oplus TG}^k \oplus (G\lef E^*) \rightrightarrows \MM:=\overbrace{TM\oplus\ldots \oplus TM}^k \oplus E^*.\]
Its algebroid is given by the following sum of VB-algebroids 
\[\AA := \overbrace{TA \oplus \cdots \oplus TA}^k \oplus (A\lef E^*) \rightarrow \MM,\]
where in the sum, we regard $A\lef E^*$ as a (pull-back) bundle over $A$.
The algebroid structure can be found in \cite{BC} for trivial representations and in \cite{CDr} in the case of non-trivial $E$.
The isomorphism between the Lie algebroid of $\GG$ and $\AA$ is given by $\psi \oplus \dots \oplus \psi \oplus id_{E^*}$.

The embedding trick (\cite{BC,CDr}) consists of regarding $E$-valued $k$-forms on $G$ and $A$ as functions on $\GG$ and $\AA$, as follows. Given an $E$-valued $k$-form $\omega$ on $G$, we consider the induced function (or local 1-cochain)  $f_\omega: \GG \dto \R$ defined by
\begin{equation}\label{eq:fom}
 f_\omega(v_1,\dots,v_k,\xi) = \xi(\omega(v_1,\dots,v_k)), \ \ v_i \in T_g G, \xi \in E^*_{\s(g)}.
 \end{equation}
Similarly, given a linear $E$-valued $k$-form $\Lambda_{(l,D)} \in \Omega^k_{\mathrm{lin}}(A,E)$ defined by a pair $(l,D)$ satisfying the Leibniz property (c.f.\ Lemma \ref{lem:one_to_one}), we define the induced function (or algebroid 1-cochain) $\delta_{(l,D)}: \AA \to \R$ by
\begin{equation}\label{eq:defdelt}
\delta_{(l,D)}(v_1,\dots,v_k,\xi) = \xi(\Lambda_{(l,D)}(v_1,\dots,v_k)), \ \ v_i \in T_a A, \xi \in E^*_{p(a)}.
\end{equation}
Note that the linearity of $\Lambda_{(l,D)}$ implies that $\delta_{(l,D)}$ is linear for the projection $\AA \to \MM$, so that it can be seen as a section of $\AA^* \to \MM$ (i.e.\ an algebroid 1-cochain).

Finally, we summarize the main properties of the embedding construction in the following compilation of results from \cite{DrE}.
\begin{proposition}\cite{DrE}\label{prop:emb}
Let $\omega$ be an $E$-valued $k$-form on $G$ and $(l,D)$ be a pair satisfying the Leibniz property defining a linear $E$-valued $k$-form on $A$.
\begin{enumerate}
\item\label{p1} $\omega$ is multiplicative if and only if $f_\omega$ is a cocycle in $\GG$.
\item\label{p2} $(l,D)$ defines a Spencer operator if and only if $\delta_{(l,D)}$ is a cocycle in $\AA$.
\item\label{p3} Let $\omega$ be multiplicative and $(l,D)$ the induced Spencer operator through \eqref{eq:Dfromomega}. Then, differentiation of the cocycle $f_\omega$ from $\GG$ to $\AA$ yields $\delta_{(l,D)}$.
\end{enumerate}
\end{proposition}

The results of \cite{DrE} hold in a generalized context in which $G\lef E^*$ and $A\lef E^*$ are replaced by the duals of a general VB-groupoid $\mathcal{V}$ and of its algebroid $\mathfrak{v}$, respectively. The specialization to our case comes from considering\footnote{In the terminology of VB-groupoids, the side-bundle of $\mathcal{V}=\t^*E$ has zero fibers and $E$ sits as the core-bundle. This implies that $\mathcal{V}^*$ has $E^*$ as side-bundle and trivial core.} $\mathcal{V}=\t^*E\rightrightarrows M$ with the groupoid structure:
\[ \bar{\s} (g,e) = \s(g), \ \bar{\t}(g,e) = \t (g) , \ \bar{\mu}((g_1,e_1),(g_2,e_2)) = (\mu(g_1,g_2), e_1 + g_1\cdot e_2).\]
See \cite[Proposition 5.9]{DrE} for item \ref{p2} above, and \cite[proof of Thm. 3.9 -- Differentiation]{DrE} for item \ref{p3} above; see also \cite[Example 3.6]{DrE} for the specialization to our case and the link to Spencer operators.
The analogue of the above proposition for forms with trivial coefficients can be found in \cite{BC}. The application of the structures on $\GG$ and $\AA$ to the description of higher order cocycles and the van Est map can be found in \cite{CDr}.

%

\subsection{The proof of Theorem \ref{thm:tubular_Spencer}}
The strategy will be to reduce the proof to the case of cocycles (c.f.\ section \ref{subsection:cocyles}) by means of the embedding trick described in the previous section.

To this end, we need to endow $\GG$ with a tubular structure. Given a tubular structure $\varphi:A\dto G$, consider the map $\hat \varphi: \AA \dto \GG$ given by
\[ \hat \varphi(v_1,\dots, v_k,\xi) = (d\varphi(v_1),\dots, d\varphi(v_k),\xi) \in T_{\varphi(a)}A \oplus \dots T_{\varphi(a)}A \oplus (\sigma^*E^*)_{\varphi(a)}\]
where $a \in A_x$, $v_i \in T_aA$ and $\xi \in E^*_x$.

\begin{lemma}\label{lma:GGtub}
The map $\hat \varphi$ defined above defines a tubular structure for $\GG$. The associated family $\hat \lambda:\R \times \GG \dto \AA$ defined by \eqref{eq:lambda_t} is given by
\[ \hat \lambda_t(v_1,\dots,v_k,\xi) = (d\lambda_t(v_1),\dots, d\lambda_t(v_k), (tg)\cdot \xi), \ \ v_i \in T_a A, \ \xi \in E^*_{q(a)}\]
where $\lambda:\R \times A \dto G$ is defined from $\varphi$ on $G$ via \eqref{eq:lambda_t}.
\end{lemma}

\begin{proof}
Since $\GG$ is given as a $k$-fold sum of $TG$'s and of $G\lef E^*$, the proof reduces to the analgous statements for each summand. First, $(a,\xi) \mapsto (\varphi(a),\xi)$ clearly defines a tubular structure on $G\lef E^*$ whose associated map $\R \times (G\lef E^*) \dto A\lef E^*$ is fully characterized by the property
\[\theta_{G\lef E^*}\Big(\frac{d}{dt}(g_t,\xi)\Big)=\Big(\theta_G\Big(\frac{d}{dt}g_t\Big),g_t\cdot \xi\Big).\]
For the $TG$ summand, we need to show that $d\varphi:TA\dto TG$ defines a tubular structure and that the associated map $\lambda_t^{d\varphi}$ via equation \eqref{eq:lambda_t} coincides with $d\lambda_t:TG\dto TA$. 
As $\varphi$ is an open embedding then $d\varphi$ is an open embedding as well; since $\sigma\circ \varphi=q$, we have that $d\varphi\circ d\sigma=dq$; and since $\varphi|_M=\mathrm{id}_M$, we have that $d\varphi|_{TM}=\mathrm{id}_{TM}$. Consider a path $a_{\epsilon}\in A$, and let $v:=\frac{d}{d\epsilon}a_{\epsilon}|_{\epsilon=0}\in TA$. Since $\sigma(\varphi(ta_{\epsilon}))=q(ta_{\epsilon})=q(a_{\epsilon})$,
is independent of $t$, by using Lemma \ref{lema:MC_TG}, we obtain
\[\theta_{TG}\Big(\frac{d}{dt}d\varphi\circ d m_t(v)\Big)=\theta_{TG}\Big(\frac{d}{dt}\frac{d}{d\epsilon}\big|_{\epsilon=0}\varphi(ta_{\epsilon})\Big)=\frac{d}{d\epsilon}\big|_{\epsilon=0}\theta_{G}\Big(\frac{d}{dt}\varphi(ta_{\epsilon})\Big).\]
For $t=0$, since $\varphi$ is a tubular structure, we obtain
\[\tau\Big(\frac{d}{dt}\Big|_{t=0}d\varphi\circ d m_t(v)\Big)=v,\]
which concludes the proof that $d\varphi$ is a tubular structure (the use of $\tau$ is necessary, because $TA$ is only isomorphic to the Lie algebroid of $TG$ as a vector bundle). For arbitrary $t$, we obtain $\theta_{TG}(\frac{d}{dt}d\varphi\circ d m_t(v))=d\lambda_t\circ d\varphi (v)$, hence taking $v = d\varphi^{-1}(w)$ we obtain $\lambda^{d\varphi}_t = d\lambda_t$ as desired.
\end{proof}

We are now ready to prove our main theorem.

\bigskip

\begin{proof}[of Theorem \ref{thm:tubular_Spencer}]
Let $(l,D)$ be a Spencer operator of degree $k$ on $A$ with values in $E$ and $\Lambda=\Lambda_{(l,D)} \in \Omega^k_{\mathrm{lin}}(A,E)$ the associated linear $k$-form (c.f. Lemma \ref{lem:one_to_one}). Denote by $\delta=\delta_{(l,D)}: \AA \to \R$ the induced function via \eqref{eq:defdelt}. By Proposition \ref{prop:emb} (item \ref{p2}), $\delta$ is a cocycle on $\AA$. As explained in section \ref{subsection:cocyles}, given the tubular structure $\hat \varphi: \AA \dto \GG$ described above, $\delta$ integrates to a local cocycle $\tilde f : \GG \dto \R$ given by the formula
\[ \tilde f (v_1,\dots,v_k,\xi) = \int_0^1 \delta(\hat \lambda_t(v_1,\dots,v_k,\xi)) \ dt, \ v_i \in T_g G, \ \xi \in E^*_{\sigma(g)}.\]
On the other hand, let $\omega$ be the $E$-valued $k$-form on $G$ defined by the equation \eqref{eq:mult-form-coeff} in the statement of the Theorem and denote by $f_\omega : \GG \dto \R$ the associated function via \eqref{eq:fom}. We claim that $\tilde f = f_\omega$. Indeed,
\begin{eqnarray*}
\tilde f (v_1,\dots,v_k,\xi) &= &\int_0^1 \delta(d\lambda_t(v_1),\dots,d\lambda_t(v_k),(tg)\cdot \xi)\ dt \\
&=&\int_0^1 \left( (tg)\cdot \xi \right) (\Lambda( d\lambda_t(v_1),\dots,d\lambda_t(v_k) ) ) \ dt \\
&=&\int_0^1 \xi ( (tg)^{-1}\cdot \Lambda( d\lambda_t(v_1),\dots,d\lambda_t(v_k) ) ) \ dt \\
&=& f_\omega (v_1,\dots,v_k,\xi).
\end{eqnarray*}
In the first line above we used the characterization of $\hat \lambda_t$ given in Lemma \ref{lma:GGtub}; in the second line we used the definition \eqref{eq:defdelt} of $\delta$ in terms $\Lambda$; in the third line we used the definition of the dual representation of $G$ on $E^*$; in the last line we used the definition \eqref{eq:fom} of $f_\omega$ in terms of $\omega$. Thus, since $\tilde f$ is a cocycle on $\GG$, so is $f_\omega$. By Proposition \ref{prop:emb} (item \ref{p1}), we conclude that $\omega$ is multiplicative. By Proposition \ref{prop:emb} (item \ref{p3}) and Lemma \ref{lem:one_to_one}, we conclude that the Spencer operator associated to $\omega$ via \eqref{eq:Dfromomega} coincides with the given $(l,D)$, i.e.\ that $\omega$ integrates $(l,D)$. Finally, to address the uniqueness, suppose that $\tilde \omega$ is another integration of $(l,D)$. By Proposition \ref{prop:emb} (item \ref{p3}), both $f_\omega$ and $f_{\tilde \omega}$ are cocycles on $\GG$ that integrate the same $\delta:\AA \to \R$. As discussed in Section \ref{subsection:cocyles}, the germ of these two cocycles around $\MM \subset \GG$ must coincide, and this directly implies that $\omega = \tilde \omega$ on points near $M\subset G$. The theorem is thus proven.
\end{proof}


 \subsection{Integration to the spray groupoid}\label{subsection:int_spray}

We discuss the integration procedure for spray groupoids. Let $A$ be a Lie algebroid and $V\in \mathfrak{X}(A)$ be a Lie algebroid spray for $A$. Consider the associated local spray groupoid $G_V\rr M$ \cite{love}. Recall that $G_V$ is an open subset of $A$ containing the zero-section in $A$, which is the unit map, the source map is the bundle projection $\sigma=q$, the target is $\tau=q\circ \phi^1_V$, the inversion is $\iota=-\phi^1_V$. To define the multiplication, we first constructed the Maurer-Cartan form of $G_V$
\[\theta:T^{\sigma}G_V\dto A,\]
which had already been introduced in \cite{Ori}. In particular, $\theta$ it is characterized by the property that it is a Lie algebroid map, by the relation (see \cite{Ori,love})
\begin{equation}\label{eq:theta_Euler2}
\theta\Big(\frac{d}{dt}ta\Big)=\phi_V^{t}(a).
\end{equation}
Finally, $\mu:{G_V}_\sigma{\times}_{\tau}G_V\dto G_V$ is defined as $\mu(a,b)=k_1$, where $k_t\in A_{q(b)}$, $t\in [0,1]$, is the solution of the ODE:
\[\frac{dk_t}{dt}=\theta_{k_t}^{-1}(\phi_V^t(a)),\ \ k_0=b.\]
Note that the identity map $\mathrm{id}:A\dto G_V$ induces a tubular structure on $G_V$. By equation (\ref{eq:theta_Euler2}), the corresponding map $\lambda_t$ (\ref{eq:lambda_t}) becomes the flow of $V$
\begin{equation}\label{eq:lambda_simple}
\lambda_t(a)=\phi^t_V(a).
\end{equation}

Consider a representation $(E,\nabla)$ of $A$. The spray condition implies that for any $a\in A$, we have that $t\mapsto \phi_V^t(a)$ is an $A$-path \cite{CF1}, i.e.\ it satisfies 
\[\frac{d}{dt}q\circ \phi_V^t(a)=dq (V_{\phi_V^t(a)})=\rho({\phi_V^t(a)}).\]
Therefore one can define the parallel transport \cite{CF1,Fe1} of the $A$-connection $\nabla$ along $\phi^{\bullet}_V(a)$:
\[T^{t,0}_{\phi_V^{\bullet}(a)}:E_{q(a)}\diffto E_{q(\phi_V^t(a))},\]
for all $t$ in the domain of the flow $\phi^{\bullet}_V(a)$. We review the details of this construction in the proof of the following result, which describes integration of representations to the spray groupoid:

\begin{lemma}
The representation $(E,\nabla)$ of $A$ integrates to the following local representation of $G_V$
\[F:G_V\dto GL(E),\ \ a\mapsto T^{1,0}_{\phi_V^{\bullet}(a)}:E_{\sigma(a)}\diffto E_{\tau(a)}.\]
\end{lemma}
\begin{proof}
The Lie algebroid $\mathfrak{gl}(E)$ of $GL(E)$ can be described as follows: as a vector bundle, the fiber of $\mathfrak{gl}(E)$ over $x\in M$ consists of linear vector fields along $E_x$, namely:
\[\mathfrak{gl}(E)_{x}:=\{v\in \Gamma(TE|_{E_x})\ : \ dm_t(v_e)=v_{te},\ \  \ \forall \ t\neq 0,\forall e\in E_x\},\]
where $m_t:E\to E$ denotes multiplication by $t\in \R$. Sections of $\mathfrak{gl}(E)$ can be seen as linear vector fields on $E$, and the Lie algebroid bracket comes from the inclusion 
$\Gamma(\mathfrak{gl}(E))\subset \mathfrak{X}(E)$. The Maurer-Cartan form of $GL(E)$ is described as follows: given a smooth family of linear isomorphisms $g_t:E_x\diffto E_{x_t}$, then  
\[\theta_{GL(E)}\Big(\frac{d}{dt}g_t\Big)=\frac{d}{dt}g_sg_t^{-1}\big|_{s=t}\in \mathfrak{gl}(E)_{x_t}.\]
The representation of $A$ can be viewed as a Lie algebroid map $f:A\to \mathfrak{gl}(E)$ covering the identity; the corresponding flat $A$-connection on $E$ (see \cite[Lemma 2.33]{CF2}) is determined by 
\[\overrightarrow{\nabla_a(e)}=[f\circ a,\overrightarrow{e}],\ \ \ a\in \Gamma(A), \ e\in \Gamma(E),\]
where $\overrightarrow{e}\in \mathfrak{X}(E)$ denoted the constant vertical vector field corresponding to a section $e\in \Gamma(E)$. 

The parallel transport $T^{t,0}_{\phi_V^{\bullet}(a)}$ along $\phi_V^t(a)$ is the solution to the ODE:
\[k_0=\mathrm{id}_{E_{q(a)}}, \ \ \theta_{GL(E)}\Big(\frac{d}{dt}k_t\Big)=f(\phi_V^t(a)).\]
By \cite[Corollary 3.17]{love}, $f$ integrates to a local Lie groupoid map $F:G_V\dto GL(E)$, determined by the condition that $F(ta)$ is the solution to the same ODE; hence the conclusion.
\end{proof}

By the proof above, we obtain that the action of $G_V$ on $E$ is given by
\begin{equation}
(ta)\cdot e=T^{t,0}_{\phi_V^{\bullet}(a)}(e).
\end{equation}
Using this and equation (\ref{eq:lambda_simple}), Theorem \ref{thm:tubular_Spencer} takes the following form in the case of spray groupoids, which for trivial coefficients specializes to Theorem \ref{theorem:spray_no_coefficeints} from the Introduction:
\begin{theorem}\label{theorem:Spencer}
Let $G_V\rightrightarrows M$ be the spray groupoid of a spray $V$ on $A$. Given an $E$-valued Spencer operator $(l,D)$ of degree $k$ on $A$, with corresponding linear form $\Lambda=\Lambda_{(l,D)}$, the formula
\begin{equation}\label{eq:mult-form-coeff}
a\mapsto \omega_a=\int_{0}^1 T_{\phi_V^{\bullet}(a)}^{0,t}\cdot (\phi_V^t)^*\Lambda_{\phi_V^t(a)}\ dt \in  \wedge^kT^*_aA\otimes E_{q(a)}
 \end{equation}
defines a local $E$-valued multiplicative $k$-form $\omega$ on $G_V$ which integrates $(l,D)$.
Moreover, the germ around $M$ of such an integration is unique.
\end{theorem}

\subsection{Some remarks}

\begin{remark}[Chain map]\label{rem:formula_chain}
For later use let us remark the following: in the case of trivial coefficients, the differentiation procedure (\ref{eq:differentiation}) is a chain map, where the differential on multiplicative forms is simply the de Rham differential, and the differential on IM-forms acts as \[d_{IM}(l,\nu)=(\nu,0).\]
\end{remark}

\begin{remark}[Multiplicative forms at units]\label{rem:omega_along}
Let $\omega$ be a multiplicative form which integrates the IM-form $(l,\nu)$. At points $x\in M\subset G$ on the unit section, under the natural decomposition $T_xG = T_xM \oplus A_x$, by using multiplicativity of $\omega$ one can show that (see \cite[Lemma 4.2]{Maria}):
\[  \omega(a_1,\ldots, a_j, v_1,\ldots,v_{k-j}) =\frac{1}{j}\sum_{i=1}^{j} (-1)^{i-1}\omega(a_i,\rho a_1,\ldots,\widehat{\rho a_{i}},\ldots,\rho a_j, v_1,\ldots,v_{k-j})\]
where $v_i \in T_xM $, $a_i \in A_x $. Therefore, by relation \eqref{eq:differentiation}, $\omega$ is determined by $l$ at the unit section:
 \begin{equation}\label{eq:omegaatzero}
  \omega(a_1,\ldots, a_j, v_1,\ldots,v_{k-j}) = \frac{1}{j}\sum_{i=1}^{j} (-1)^{i-1}l(a_i)(\rho a_1,\ldots,\widehat{\rho a_{i}},\ldots,\rho a_j, v_1,\ldots,v_{k-j}).
 \end{equation}
Applying this to $d_{IM}(l,\nu)=(\nu,0)$, and Remark \ref{rem:formula_chain}, we obtain a similar result for $d\omega$:
 \begin{equation*}
  d\omega(a_1,\ldots, a_j, v_1,\ldots,v_{k+1-j}) = \frac{1}{j}\sum_{i=1}^{j} (-1)^{i-1}\nu(a_i)(\rho a_1,\ldots,\widehat{\rho a_{i}},\ldots,\rho a_j, v_1,\ldots,v_{k+1-j}).
 \end{equation*}
\end{remark}

\begin{remark}[Exact multiplicative forms]\label{rem:mult_exact}
For the Lie algebroid $TM$ any IM-$k$-form is of the form $(\varpi^{\flat},(d\varpi)^{\flat})$, for a unique $k$-form $\varpi\in \Omega^k(M)$, where for $\eta\in\Omega^k(M)$, we denote $\eta^\flat:TM\to\wedge^{k-1}T^*M,\ v\mapsto \iota_v\eta$. For any Lie algebroid $A$ over $M$, one pulls back this IM-form via the anchor and obtain the IM-$k$-form $(\varpi^{\flat}\circ\rho, (d\varpi)^{\flat}\circ\rho)$ on $A$. The corresponding global construction is as follows. If $G\rr M$ is a local Lie groupoid integrating $A$, then, for $\varpi\in \Omega^k(M)$, $\omega:=\tau^*\varpi-\sigma^*\varpi$ is a local multiplicative form which integrates $(\varpi^{\flat}\circ\rho, (d\varpi)^{\flat}\circ\rho)$. Let us verify that this is what Theorem \ref{thm:tubular_Spencer} predicts. First, note that, for any $\eta\in \Omega^{p}(M)$ and any spray $V$ on $A$, we have that:
\[i_Vq^*(\eta)=(\eta^{\flat}\circ \rho)^*(\alpha^{p-1}_{\mathrm{can}}).\]
This implies that the linear form corresponding to $(\varpi^{\flat}\circ\rho, (d\varpi)^{\flat}\circ\rho)$ is
\[\Lambda=d i_Vq^*\varpi+i_Vq^*d\varpi=\mathcal{L}_Vq^*\varpi.\]
Therefore, we obtain the expected result:
\begin{align*}
\omega&=\int_0^1(\phi_V^t)^*\Lambda\ dt=\int_0^1(\phi_V^t)^*\Lie_Vq^*\varpi\ dt=
\int_0^1\frac{d}{dt}(\phi_V^t)^*q^*\varpi\ dt=\\
&=(q\circ \phi_V^1)^*\varpi-q^*\varpi=\tau^*\varpi-\sigma^*\varpi.
\end{align*}
\end{remark}

\begin{remark}[Mutiplicative forms on the space of $A$-paths]
Recall from \cite{CF1} that the \textbf{Weinstein groupoid} associated to a Lie algebroid $A$ is $G(A)=P(A)/\mathcal{F}\rr M$, where $P(A)$ denotes the space of $C^1$-$A$-paths, and $\mathcal{F}$ is the $A$-homotopy foliation. Given a spray $V$ on $A$, there is an associated exponential map 
\begin{equation}\label{eq:hatexp}
\exp_V: A \dto P(A),\ \  a \mapsto (t\mapsto \phi_V^t(a))_{t\in [0,1]},
\end{equation}
and the spray groupoid $G_V$ is such that $\exp_V$ induces a local groupoid map. The origin of the formula for integrating IM-forms becomes transparent using this example. Namely, an IM-$k$-form $(l,\nu)$ on $A$ induces naturally a differential $k$-form on $P(A)$:
\[ \hat{\omega} = \int_0^1 ev^*\Lambda_{(l,\nu)}\ dt,\]
where $ev: P(A) \times [0,1] \to A$ denotes $(a, t)\mapsto a_t$. It follows that the local form $\omega$ on $G_V$ from Theorem \ref{theorem:spray_no_coefficeints} coincides with $\exp_V^*\hat{\omega}$. The fact that $\omega$ is multiplicative and integrates $(l,\nu)$ can be seen as a consequence of the fact that $\hat{\omega}$ is basic relative to $A$-homotopy foliation $\mathcal{F}$ on $P(A)$ (see also \cite[Remark 4 after Theorem 2]{BC}). This argument is explained in detail in \cite{BX}. However, the proof of Theorem \ref{theorem:spray_no_coefficeints} that we shall provide here is independent of these observations, and does not involve any infinite dimensional differential geometry; it is based on the general properties of the spray groupoid construction.
\end{remark}

\begin{remark} [Characterizing global multiplicative forms]
 Let $H\tto M$ be a (global) Lie groupoid which is source connected and simply connected, and let $\omega$ a multiplicative form on $H$ with trivial coefficients (for simplicity). Following \cite[Appendix A]{BaGu}, the restriction $\omega|_G$ to a source connected local Lie subgroupoid $G\subset H$ over $M$ fully determines $\omega$. Moreover, any multiplicative form on $G$ extends uniquely to a multiplicative form on $H$. Thus, the formula of Theorem \ref{theorem:spray_no_coefficeints} uniquely characterizes any multiplicative form on any source $1$-connected Lie groupoid $H\tto M$.
\end{remark}

\begin{remark} [Integrating local multiplicative multivectors]
Following \cite[Section 6]{BC}, the description of multiplicative multivectors is quite parallel to our discussion of multiplicative forms. In this case, the role of the tangent local groupoid $TG\rightrightarrows TM$ is replaced by the {\bf cotangent local (VB-)groupoid} $T^*G\rightrightarrows A^*$ (with cotagent Lie algebroid $T^*A \to A^*$). A local multiplicative multivector field $\Pi$ on $G\rightrightarrows M$ corresponds to a local $1$-cocycle on $k$-copies of $T^*G$ and its infinitesimal counterpart $\pi$ to a $1$-cocycle on $k$-copies of $T^*A$ (for more details, see \cite{BC,ILX}). To provide a formula for $\Pi$ in terms of the infinitesimal $\pi$ analogous to that of Theorem \ref{thm:tubular_Spencer}, one can try to proceed as in the case of forms above and fix a tubular structure $\widehat{\varphi}:T^*A\dto T^*G$, with associated family $\widehat{\lambda}_t:T^*G\dto T^*A$. Unlike the case of $TG$, there is no natural way of inducing a tubular structure on $T^*G$ from a tubular structure $\varphi:A\dto G$ (without making any additional non-canonical choices). Nevertheless, this construction leads to the following formula (recall the integration of cocycles from Subsection \ref{subsection:cocyles})
\[\Pi(\alpha_1,\dots,\alpha_k)= \int_0^1  \pi(\widehat{\lambda}_t(\alpha_1), \ldots,\widehat{\lambda}_t(\alpha_k))\ dt , \ \alpha_i \in T^*_g G,\]
which can be shown to indeed define a local multipicative multivector $\Pi$ integrating $\pi$. We also remark that the arguments for forms and multivectors can also be combined to yield formulas for {\bf local multiplicative tensor fields} (see \cite{BDtensor}).
\end{remark}

\section{Applications}\label{sec:appl}

Many geometric structures, for example Poisson, Dirac, Jacobi, generalized complex structures etc., can be encoded by a Lie algebroid endowed with a certain IM-form or, more generally, with a Spencer operator. Here we apply our general results to provide explicit local integrations of such structures to local multiplicative forms on the corresponding spray groupoids. As a byproduct, we reprove the global formulas for realizations of the following structures: Poisson \cite{CM}, Dirac \cite{IP}, and Nijenhuis-Poisson (including holomorphic Poisson) \cite{BX,Peta}. We work out the integration of Jacobi structures by local contact groupoids, which involve non-trivial representations.

\subsection{Local integration of Poisson structures}\label{subsec:Poisson}

In \cite{CDW}, it is proven that a symplectic realization of a Poisson manifold, together with a fixed Lagrangian section, determine a local symplectic groupoid on the total space of the realization that integrates the Poisson manifold. This was one of the first general constructions of local Lie groupoids. By applying the ``spray method'' to this case, we give an explicit construction of a local symplectic spray groupoid integrating a Poisson manifold.
\vspace*{0.3cm}

Let $(M,\pi)$ be a Poisson manifold. The associated \textbf{cotangent Lie algebroid} has total space $T^*_{\pi}M=T^*M \to M$, has anchor map given by $\pi$ turned into a bundle map $\pi^{\sharp}:T_{\pi}^*M\to TM$, and has Lie bracket given by $[a,b]_{\pi}=\mathcal{L}_{\pi^{\sharp}(a)}b-i_{\pi^{\sharp}(b)}da$, for $a,b\in \Omega^1(M)$.

The cotangent Lie algebroid $T^*_{\pi}M$ carries a canonical closed IM-2-form given by $(-\mathrm{id},0)$, with corresponding Spencer operator $(-\mathrm{id},-d)$, and with corresponding linear form the canonical symplectic structure $\omega_{0}$ on $T^*M$ (see e.g.\ \cite{BC}).

Given a spray $V$ for $T^*_\pi M$, also called a \textbf{Poisson spray}, we construct the associated spray groupoid $G_V\rightrightarrows M$. By Theorem \ref{theorem:spray_no_coefficeints}, $G_V$ carries a local multiplicative $2$-form given by
\begin{equation}\label{eq:like_in_CM}
\omega = \int_0^1 (\phi^t_V)^*\omega_{0}\ dt.
\end{equation}
We claim that, in a neighborhood of $M$,
\[(G_V,\omega)\rr (M,\pi)\]
is a local symplectic groupoid integrating the Poisson manifold $(M,\pi)$, in the sense of \cite{CDW}. First, $\omega$ is obviously closed, and even exact, with (non-multiplicative) primitive $-\int_0^1 (\phi^t_V)^*\alpha_0 \ dt$, where $\alpha_0$ denotes the tautological 1-form on $T^*M$. At points on the unit-section $x\in M\subset G_V$, using the canonical decomposition $T_x(T^*M)=T_xM\oplus T^*_xM$, equation \eqref{eq:omegaatzero} gives
	\begin{equation}\label{eq:omega}
\omega_{x}(v+a,w+b)=\langle b\ |\ v\rangle-\langle a\ | \ w\rangle+\pi(a,b), \end{equation}
for all $v+a,w+b\in T_{x}G_V$. In particular, this formula shows that $\omega$ is non-degenerate along $M$, and hence, it is so in a neighborhood of $M$. Thus, $(G_V,\omega)\rr M$ is indeed a local symplectic groupoid. Following \cite{CDW}, $\omega^{-1}$ is $q$-projectable to a Poisson structure on $M$, and it follows from the expression (\ref{eq:omega}) at points of $M$ that this Poisson structure is $\pi$. This proves the claim.

Being a local symplectic groupoid integrating $\pi$ implies that $q:(T^*M,\omega)\dto (M,\pi)$ defines a symplectic realization. This fact already appears in \cite{CM} and the proof given there is a direct but tedious calculation, whereas the one given here is Lie theoretic in nature, as it exploits the local Lie groupoid $G_V$ and the multiplicativity of $\omega$ (another short, Dirac-geometric proof of the realization property is found in \cite{IP}).

On the other hand, as noticed in \cite{CM}, the local Lie groupoid structure on $G_V$ is determined by the symplectic realization $q:(G_V,\omega)\to (M,\pi)$: first, in a neighborhood of $M$, the Maurer-Cartan form $\theta$ corresponding to the spray $V$ can be given in terms of $\omega$ by:
\begin{equation}\label{eq:MC_for_Poisson}
\tau^*(\theta_a(v))=-i_v\omega_a\in T^*_{a}G_V, \ \ \ \ v\in T^{\sigma}_aG_V.
\end{equation}
This follows by the multiplicativity of $\omega$: for any $w\in T_aG_V$, using (\ref{eq:omega}) and writing $\theta_a(v)=d\mu(v,0_{a^{-1}})$, and $d\tau(w)=d\mu(w,d\iota(w))$, we obtain:
\begin{align*}
\tau^*(\theta_a(v))(w)&=\langle \theta_a(v)\ | \ d\tau(w)\rangle=\omega_{\tau(a)}(d\tau(w),\theta_a(v))=\\
 &=\omega_a(w,v)+\omega_{a^{-1}}(d\iota(w),0_{a^{-1}})=-\omega_a(v, w),
\end{align*}
which proves (\ref{eq:MC_for_Poisson}). Using formula (\ref{eq:MC_for_Poisson}), we obtain a simpler description of the ODE from \cite{love} describing the multiplication of $G_V$: for $(a,b)\in {G_V}_{\sigma}\times_{\tau}G_V$, close enough to $M$, the curve $k_t=\mu(ta,b)$, with $t\in [0,1]$, is the solution of the ODE
	\begin{equation}
\label{eq:Poisson} \frac{d}{dt}k_t =  -\Pi^\sharp_{k_t}( \tau^*(\phi^t_V(a))), \ \ k_0=b,
\end{equation}
where $\Pi$ denotes the inverse of $\omega$, which is defined on a neighborhood of $M$ in $G_V$.

\subsection{Closed IM-2-forms on Poisson manifolds}\label{sec:closed2forms}

Consider the same setting as in the previous subsection: $(M,\pi)$ is a Poisson manifold, $V$ is a spray on the cotangent Lie algebroid $T^*_{\pi}M$, and $(G_V,\omega)\rr (M,\pi)$ is the corresponding local symplectic spray groupoid.

A closed IM-2-form on $T^*_{\pi}M$, is given by a vector bundle map $l:T^*M\to T^*M$, such that:
\[\pi(l(a),b)=\pi(a,l(b)),\  \ \ l([a,b]_{\pi})=\Lie_{\pi^{\sharp}(a)}l(b)-i_{\pi^{\sharp}(b)}dl(a),\]
for all $a,b\in \Omega^1(M)$.

The linear 2-form associated to $(-l,0)$, is given by:
\[\Lambda_{(-l,0)}=l^*(\omega_0)\in \Omega^2(T^*M),\]
where $\omega_0$ is the canonical symplectic structure on $T^*M$. Thus, the multiplicative closed $2$-form on the spray groupoid $G_V$ corresponding to $(-l,0)$ is given by
\[\omega_{(-l,0)}=\int_0^1(l\circ \phi_V^t)^*\omega_0\ dt.\]
Since the symplectic $\omega$ is non-degenerate around $M$, we can write:
\[\omega_{(-l,0)}(u,v)=\omega_L(u,v):=\omega(Lu,v), \ \ u,v\in TG_V,\]
for a unique vector bundle map $L:TG_V\to TG_V$. The fact that $\omega$ and $\omega_L$ are multiplicative implies that $L$ is multiplicative, i.e.\ it is a groupoid morphism (the Lie algebroid morphism induced by $L$ is the \textbf{complete lift} of $l$ to the cotangent bundle \cite[Definition 7.1]{BX}, see also \cite{BDtensor}):
\[
\xymatrixrowsep{0.4cm}
\xymatrixcolsep{1.2cm}
\xymatrix{
 TG_V \ar@{-->}[r]^{L}\ar[d] & TG_V\ar[d] \\
 TM\ar@{->}[r]^{l} & TM}\]
where the dual of $l$ was denote by the same symbol $l=l^{t}:TM\to TM$. To see that the base map is indeed $l$, we calculate $\omega_{L}=\omega_{(-l,0)}$ along the zero-section, using Remark \ref{rem:omega_along}, and obtain:
\[\omega_{L}(v+a,w+b)=\langle v\ |\ l(b)\rangle-
\langle w\ |\ l(a)\rangle-\pi(l(a),b)=\omega(l(v)+l(a),w+b),\]
for all $v+a,w+b\in T_xG_V$, $x\in M$. Hence, along the zero-section,
\begin{equation}\label{eq:L_along_zero}
L(v+a)=l(v)+l(a).
\end{equation}
Finally, denote by $\Pi:=\omega^{-1}$. For each $k\geq 0$, consider the bivector fields:
\[\Pi_{L^k}\in \X^2(G_V), \ \ \Pi^{\sharp}_{L^k}:=\Pi^{\sharp}\circ L^k,\ \ \
\pi_{l^k}\in \X^2(M), \ \ \pi^{\sharp}_{l^k}:=\pi^{\sharp}\circ l^k.\]
Since $\sigma_*(\Pi)=\pi$, and $L$ and $l$ are $\sigma$ related, it follows that also:
\[\sigma_*(\Pi_{L^k})=\pi_{l^{k}}\ \ \ \  \textrm{(and similarly,}\ \  \tau_*(\Pi_{L^k})=-\pi_{l^{k}}\textrm{).}\]
In general, these bivector fields are not Poisson.

\subsubsection{The case of invertible IM-forms}

We discuss now the case when $l:T^*M\to T^*M$ is invertible. By (\ref{eq:L_along_zero}), this is equivalent to $L$ being invertible around $M$, and is also equivalent to $\omega_L$ being symplectic. By the argument above, also for $k<0$ we have that $\sigma_*(\Pi_{L^{k}})=\pi_{l^{k}}$. Moreover, the following hold:
\begin{itemize}
\item $\pi_{l^{-1}}$ is Poisson and $(G_V,\omega_{L})\rr (M,\pi_{l^{-1}})$ is a symplectic groupoid.
\item The vector field $l_*V=dl\circ V\circ l^{-1}$ is a Poisson spray for $\pi_{l^{-1}}$. If $(G_{l_*V},\omega_{l_*V})\rr (M,\pi_{l^{-1}})$ denotes the local symplectic spray groupoid of $l_*V$, then \[l:(G_{V},\omega_{L})\dto (G_{l_*V},\omega_{l_*V})\]
    gives an isomorphism between germs of local symplectic groupoids.
\end{itemize}

Clearly $(G_V,\omega_{L})$ is a symplectic groupoid. Since $(\omega_{L})^{-1}=\Pi_{L^{-1}}$, and $\sigma_*(\Pi_{L^{-1}})=\pi_{l^{-1}}$, the first part follows. By a direct verification, $l_{*}V$ is indeed a Poisson spray for $\pi_{l^{-1}}$, and clearly, its flow is $\phi_{l_*V}^t=l\circ \phi_V^t\circ l^{-1}$. Since $l$ intertwines the two sprays, \cite[Example 3.18]{love} implies that $l$ is a local Lie groupoid map between the two spray groupoids $l:G_{V}\dto G_{l_*V}$. Finally, note that $l$ is a local symplectomorphism:
\[\omega_{L}=\int_0^1(l\circ \phi_V^t)^*\omega_0\ dt=
\int_0^1(\phi_{l_*V}^t\circ l)^*\omega_0\ dt=
l^*\big(\int_0^1(\phi_{l_*V}^t)^*\omega_0\ dt\big)=l^*(\omega_{l_*V}).\]
This completes the proof of the claims made above.

In particular we have obtained that the projection $\sigma=q:T^*M=G_V\to M$ gives two symplectic realizations:
\[q:(\mathcal{U},\omega)\to (M,\pi),\ \ q:(\mathcal{U},\omega_L)\to (M,\pi_{l^{-1}}),\]
\[\omega=\int_0^1(\phi_V^t)^*\omega_{0}\ dt,\ \ \omega_L=\int_0^1(\phi_V^t)^*l^*(\omega_{0})\ dt,\]
for some neighborhood $\mathcal{U}\subset T^*M$ of $M$. This gives a different proof of \cite[Theorem 3.3]{Peta}, but without the assumption that $l$ be a Nijenhuis tensor (!) (see also Subsection \ref{subsubsec:Nijenhuis}).

\subsubsection{Deformations of a multiplicative symplectic structure}

A closed IM-$2$-form $l$ induces a \textbf{deformation of the multiplicative symplectic structure} $\omega$ of $G_V$, namely, the one-parameter family of closed multiplicative 2-form on $G_V\rr M$:
\[\omega_{L_{\epsilon}}=\omega+\omega_{\epsilon L},\]
defined for small $\epsilon$, which corresponds to $l_{\epsilon}=\mathrm{Id}+\epsilon l$, and therefore $L_{\epsilon}=\mathrm{Id}+\epsilon L$. We will assume that $l_{\epsilon}$ is invertible (which always holds locally on $M$, for small $\epsilon$), which implies that $L_{\epsilon}$ is invertible, and also that $\omega_{L_{\epsilon}}$ is symplectic. As discussed above,
\[(G_V,\omega_{L_{\epsilon}})\rr (M,\pi_{l^{-1}_\epsilon})\]
is a symplectic groupoid. Note that Taylor expansion in $\epsilon$ of the Poisson structure is given by:
\[\pi_{l_{\epsilon}^{-1}}=\pi-\epsilon\pi_l+\epsilon^2\pi_{l^2}-\epsilon^3\pi_{l^3}+\ldots. \]
Then $\pi_l$, the speed at $\epsilon=0$, is a Poisson-cocycle: $[\pi,\pi_l]=0$. 

Let us discuss the case of \textbf{trivial deformations}, i.e.\ assume that $\pi_l=-[\pi,X]$, for some vector field $X$. This condition is equivalent to $(X,l)$ being an IM-1-form, i.e.\ to $l$ being exact; moreover, since the integration of $IM$-forms is a chain complex isomorphism to germs of multiplicative forms (Remark \ref{rem:formula_chain}), the condition is also equivalent to $\omega_L$ admitting a local multiplicative primitive $\eta$. By Theorem \ref{theorem:spray_no_coefficeints}, this multiplicative 1-form integrating $(X,l)$ is given by
\[\eta=\int_0^1(\phi_V^t)^*(l^*\alpha_0)\ dt+d\int_0^1(\phi_V^t)^*\widetilde{X}\ dt,\]
where $\alpha_0$ is the tautological 1-form on $T^*M$ and $\widetilde{X}\in \C(T^*M)$ is $X$ viewed as a linear function. Then $\omega_L=-d\eta$, and so $\omega_{L_{\epsilon}}=\omega-\epsilon d\eta$. Consider the local, time-dependent vector field $Y_{\epsilon}$, defined by
\[i_{Y_{\epsilon}}\omega_{L_{\epsilon}}=\eta.\]
Then $Y_{\epsilon}$ is multiplicative, hence its flow $\phi_{Y_{\bullet}}^{\epsilon}$ is a local groupoid map; and by the usual Moser argument, $(\phi_{Y_{\bullet}}^{\epsilon})^*\omega_{L_\epsilon}=\omega$. Thus we obtain a local symplectic groupoid isomorphism:
\[\phi_{Y_{\bullet}}^{\epsilon}:(G_V,\omega)\dto(G_V,\omega_{L_{\epsilon}}).\]
Since $Y_{\epsilon}$ is multiplicative, it is $\sigma$-projectable to a vector field $X_{\epsilon}$ on $M$ whose flow, when defined, is a Poisson isomorphism:
\[\phi_{X_{\bullet}}^{\epsilon}:(M,\pi)\to (M,\pi_{l^{-1}_\epsilon}).\]
In fact, by calculating $\omega_{L_{\epsilon}}$ and $\eta$ along the zero-section with the aid of Remark \ref{rem:omega_along},
\[\omega_{L_{\epsilon}}(v+a,w+b)=\langle v|l_{\epsilon}b\rangle-
\langle w|l_{\epsilon}a\rangle-\pi(l_{\epsilon}a,b),\ \ \eta(v+a)=\langle X_x|a\rangle,\]
for all $v+a,w+b\in T_xG_V$, $x\in M$, one obtains that $Y_{\epsilon}|_M=l_{\epsilon}^{-1}\circ X$. Hence, $X_{\epsilon}=l^{-1}_{\epsilon}(X)$.

Another interesting case are \textbf{gauge transformations}; i.e.\ when $l=\varpi^{\flat}\circ\pi^{\sharp}$, for a closed 2-form $\varpi$ on $M$. In this case, as discussed in Remark \ref{rem:mult_exact}, we have that $(-l,0)$ integrates to $\omega_{L}=\sigma^*(\varpi)-\tau^{*}(\varpi)$. We obtain the deformation of symplectic groupoids $(G_V,\omega_{L_{\epsilon}})\rr (M,\pi_{l^{-1}_{\epsilon}})$, \[ \textrm{where}\ \ \omega_{L_{\epsilon}}=\omega+\epsilon(\sigma^*(\varpi)-\tau^{*}(\varpi))\ \ \textrm{and}\ \  \pi_{l_\epsilon^{-1}}^{\sharp}=\pi^{\sharp}\circ (\mathrm{Id}+\epsilon\varpi^{\flat}\circ \pi^{\sharp})^{-1}\]
is the gauge transformation of $\pi$ by the closed 2-form $\epsilon \varpi$.

\subsubsection{Local integration of Poisson-Nijenhuis structures}\label{subsubsec:Nijenhuis}

Here, we recover the results on local integrations and realizations of Poisson-Nijenhuis structures \cite{BX,Peta,StiXu} using the spray groupoid perspective.

As before, let $l$ be a closed IM-2-form on $T^*_{\pi}M$. The \textbf{Nijenhuis torsion} of $l$ is defined by:
\[T_l\in \Gamma(\wedge^2T^*M\otimes TM), \] \[ T_l(u,v):=[l(u),l(v)]-l([l(u),v]+[u,l(v)])+l^2([u,v]), \ \ u,v\in \X(M),\]
where $l$ and its transposed are denoted the same. In our setting, this tensor appears as follows (the result below is inspired by results in \cite{CrGCS}):
\begin{lemma}\label{lema:torsion}
Consider the multiplicative 2-form on $G_V$, $\omega_{L^2}(u,v)=\omega(L^2u,v)$. The IM-2-form corresponding to $\omega_{L^2}$ is $(-l^2,-T_l)$; moreover, the following three equations are equivalent:
\[(1)\ \ d\omega_{L^2}=0;\ \ \ \ \ \ \ \ (2)\ \ T_l=0;\ \ \ \ \ \ \ \ (3)\ \  T_L=0.\]
\end{lemma}
\begin{proof}
The result is based on the calculations from \cite{CrGCS}. First, we note a general formula: for any 2-form $\Omega\in \Omega^2(N)$ and any $(1,1)$ tensor $A:TN\to TN$ which satisfy $\Omega(AX,Y)=\Omega(X,AY)$, we have that \cite[Lemma 2.9]{CrGCS}
\begin{equation}\label{eq:Marius}
i_{T_A(u,v)}\Omega=(i_v\circ i_{Au}+i_{Av}\circ i_u)d\Omega_A-i_{Av}\circ i_{Au}d\Omega-i_{v}\circ i_ud\Omega_{A^2},
\end{equation}
for all $u,v\in TN$, where $\Omega_{A^k}(u,v):=\Omega(A^ku,v)$.
Applying (\ref{eq:Marius}) to the multiplicative symplectic structure $\omega$ on $G_V$, and to the $(1,1)$ tensor $L$, we obtain that:
\[i_{T_L(u,v)}\omega=-i_{v}\circ i_ud\omega_{L^2}.\]
In particular, this shows that (1) and (3) are equivalent. Denote by $(\underline{l},\underline{\nu})$ the IM-2-form corresponding to $\omega_{L^2}$. By the differentiation procedure (\ref{eq:differentiation}), and by the explicit description of $\omega$ and $L$ along the zero-section, we have that:
\[\underline{l}(a)(v)=\omega_{L^2}(a,v)=\omega(l^2(a),v)=-\langle l^2a\ | \ v\rangle,\]
\[\underline{\nu}(a)(v,w)=d\omega_{L^2}(a,v,w)=-\omega(T_L(v,w),a)=-\omega(T_l(v,w),a)=-T_l(a)(v,w),
\]
where we have also used that $T_L(u,v)=T_l(u,v)$ for $u,v$ tangent to the unit section. This implies that $\omega_{L^2}$ integrates $(-l^2,-T_l)$. Finally, by Remark \ref{rem:formula_chain}, $\omega_{L^2}$ is closed iff $T_l=0$.
\end{proof}

Assume that $(\pi,l)$ is a \textbf{Poisson-Nijenhuis structure}, i.e.\ $\pi$ is a Poisson structure and $l$ is a closed IM-2-form on $T^*_{\pi}M$ whose Nijenhuis torsion vanishes. The above shows that
\[(G_V,\omega,L)\rr (M,\pi,l)\]
is a \textbf{local symplectic Nijenhuis groupoid} \cite{StiXu}, i.e.\ $(G_V,\omega)$ is a local symplectic groupoid, $L$ is multiplicative and $(\Pi=\omega^{-1},L)$ is a Poisson-Nijenhuis structure inducing on the base the Poisson-Nijenhuis structure $(\pi,l)$. By the general theory of Poisson Nijenhuis structures \cite{Yvette}, the bivector fields $\pi_{l^k}$, $k\geq 0$, are all Poisson structures, which pairwise commute; and similarly, at the groupoid level, the local multiplicative bivector fields $\Pi_{L^k}$, $k\geq 0$, are are all Poisson structures, which pairwise commute, and by the discussion in the beginning of Section \ref{sec:closed2forms}, the source map is a Poisson map for all these structures (resp.\ the target map is anti-Poisson):
\[\sigma:(G_V,\Pi_{L^k})\rmap(M,\pi_{l^k}),\ \ k\geq 0.\]
On the other hand, by using relation (\ref{eq:Marius}) for $\Omega=\omega_{L^k}$, one proves inductively that $d\omega_{L^{k}}=0$ for all $k\geq 0$. So, all these forms are multiplicative and closed, and their infinitesimal IM-forms are $(-l^k,0)$. This shows that:
\[\omega_{L^k}=\int_{0}^1(\phi_V^t)^*(l^k)^*\omega_0\ dt.\]
If $l$ is invertible, the above discussion holds for all $k\in \mathbb{Z}$, and we obtain symplectic realizations:
\[ \ \ \ \ q=\sigma:(G_V,\omega_{L^{k}})\rmap (M,\pi_{l^{-k}}), \ \ \ \ k\in \mathbb{Z}.\]

\begin{remark}
In \cite{BX}, the authors consider further \textbf{holomorphic Poisson structures}, i.e.\ Poisson Nijenhuis structures $(M,\pi,j)$ such that $j^2=-\mathrm{Id}$. In this case, the local symplectic Nijenhuis groupoid $(G_V,\omega,J)$ which is obtained by the above procedure is automatically a local \textbf{holomorphic symplectic} groupoid. This follows by integrating the closed IM-2-form $(-j^2,0)=(\mathrm{Id},0)$, and conclude that $\omega_{J^2}=-\omega$, hence $J^2=-\mathrm{Id}$.
\end{remark}

\subsubsection{Local integration of generalized complex structure}

Another setting where closed IM-2-forms on Poisson manifolds appear are generalized complex manifolds \cite{Gualtieri}. Let us describe the integration of these structures to the corresponding spray groupoid, in the spirit of \cite{CrGCS} (see also \cite{BaGu} for a more complete discussion of integration of generalized complex structures).

Following \cite{CrGCS}, a \textbf{generalized complex structure} can be described by a triple $(\pi,l,\varpi)$, where $\pi$ is a Poisson structure on $M$, $l$ is a closed IM-2-form on $T^*_{\pi}M$, and $\varpi$ is a 2-form on $M$ satisfying:
\[l^2+\pi^{\sharp}\varpi^{\flat}=-\mathrm{Id},\ \ T_l(u,v)=\pi^{\sharp}(i_{v}i_u d\varpi),\]
\[l\circ \varpi=\varpi\circ l, \ \ d\varpi_{l}(u,v,w)=d\varpi(lu,v,w)+d\varpi(u,lv,w)+d\varpi(u,v,lw),\]
where, as usual, $\varpi_l(u,v):=\varpi(l(u),v)$, and $l$ and its transpose are denoted by the same symbol.

Consider the local integration $(G_V,\omega,\omega_L)$ of $(M,\pi,l)$, as in the beginnin of Section \ref{sec:closed2forms}. Following \cite{CrGCS}, the first set of equations has an interpretation in terms of IM-2-forms (for the second, we are not aware of such an interpretation); namely, the IM forms \[-(\mathrm{Id},0),\ \  -(l^2,T_l)\ \ \textrm{and}\ \  (\varpi^{\flat}\circ \pi^{\sharp}, d\varpi^{\flat}\circ \pi^{\sharp})\]
integrate to the multiplicative forms $\omega$, $\omega_{L^2}$ (Lemma \ref{lema:torsion}) and $\tau^*(\varpi)-\sigma^*(\varpi)$ (Remark \ref{rem:mult_exact}), respectively; therefore the first conditions are equivalent to:
\[\omega+\omega_{L^2}=\tau^*(\varpi)-\sigma^*(\varpi).\]
The triple $(\Pi=\omega^{-1},L,\sigma^*(\varpi)-\tau^*(\varpi))$ is a local multiplicative generalized complex structure on the spray groupoid $G_V$, and which is a local integration in the sense of \cite{CrGCS} of $(\pi,l,\varpi)$.

\subsection{Local integration of twisted Dirac structures}
Generalizing Poisson structures, Dirac structures provide another example of Lie algebroids which carry natural IM 2-forms \cite{BC,BCWZ}. We discuss their local integrations obtained from sprays.

Fix a closed 3-form $H\in \Omega^3(M)$. Recall \cite{BCWZ} that an \textbf{$H$-twisted Dirac structure} is a subbundle $L\subset TM\oplus T^*M=:\TT M$ over $M$ which is lagrangian with respect to the natural symmetric pairing on $\TT M$ and whose space of sections $\Gamma(L)$ is involutive with respect to the $H$-twisted Courant bracket on $\Gamma (\TT M)=\X^1(M) \oplus \Omega^1(M)$:
\[\Cour{v+\alpha,w+\beta}_{H} = [v,w]+\Lie_v \beta - i_w d\alpha+i_wi_vH, \ \ v,w \in \X^1(M), \alpha,\beta \in \Omega^1(M).\]
Then $(L,\Cour{\cdot,\cdot}_{H},\rho)$ is a Lie algebroid with anchor $\rho(v+\alpha)=v$, which carries the canonical IM-2-form $(-\br, H^{\flat}\circ \rho)$ given by:
\begin{align*}
\br:L\rmap& T^*M,\ \ \ \  v+\alpha\mapsto\alpha,\\
H^{\flat}\circ \rho:L\rmap \wedge^2 &T^*M,\ \ \ \   v+\alpha\mapsto i_{v}H.
\end{align*}
The linear form on $L$ corresponding to $(-\br, H^{\flat}\circ \rho )$ is given by
\[\Lambda_{(-\br, H^{\flat}\circ \rho )}=\br^*\omega_0+\Lambda_H\in \Omega^2_{\mathrm{lin}}(L), \ \ \textrm{where}\ \Lambda_H:=\widetilde{H^{\flat}\circ \rho}.\]

Let $V$ be a spray on $L$ with corresponding local spray groupoid $G_V\rr M$. Consider the local multiplicative 2-form on $G_V$ integrating $(-\br,H^{\flat}\circ \rho )$:
\[\omega=\int_0^1(\phi^t_V)^*(\br^*\omega_0+\Lambda_H)\ dt.\]
Since integration is a chain map (Remark \ref{rem:formula_chain}), the multiplicative $3$-form $d\omega$ integrates the IM-$3$-form $(H^{\flat}\circ \rho,0)$, therefore, by Remark \ref{rem:mult_exact}, we obtain that $\omega$ is \textbf{relatively $H$-closed}, i.e.\
\[d\omega=\tau^*H-\sigma^*H.\]

Using the canonical decomposition $T_xL=T_xM\oplus L_x$, for $x\in M\subset L$, equation \eqref{eq:omegaatzero} and that $L$ is isotropic, we obtain that
	\begin{equation}\label{eq:omega_1}
\omega_{x}((v_1,w_1+\alpha_1),(v_2,w_2+\alpha_2))=\langle \alpha_2\ | \ v_1\rangle-\langle\alpha_1\ |\ v_2+w_2\rangle,
\end{equation}
for all $(v_1,w_1+\alpha_1),(v_2,w_2+\alpha_2)\in T_{x}G_V$. This formula can be used to show that, along the unit section, $\omega$ satisfies the following so-called \textbf{robustness condition} \cite{BCWZ}
\[\ker(\omega_a)\cap \ker(d_a\sigma) \cap \ker(d_a\tau) = 0,\ \  a \in G_V.\]
This condition is open, therefore it is satisfied around the unit section. Thus, $(G_V,\omega,H)\rightrightarrows M$ defines a local version of the \textbf{presymplectic groupoids} defined in \cite[Definition 2.1]{BCWZ}. The local groupoid version of \cite[Theorem 2.2]{BCWZ} also holds, and therefore $\sigma$ pushes the graph of $\omega$ forward to an $H$-twisted Dirac structure on $M$, and this Dirac structure is $L$\footnote{Our conventions differ slightly from \cite{BCWZ}: in \emph{loc.\ cit.\ }the presymplectic structure of a presymplectic groupoid is the negative of our presymplectic structure, and this makes the target map forward Dirac.}, as it can be straightforwardly checked at points of the unit section using (\ref{eq:omega_1}), just as in the Poisson case. Moreover, the source-target map becomes a \textbf{presymplectic realization} \cite{BCWZ}:
\begin{equation}\label{eq:presymplectic_realization}
\sigma\times\tau: (G_V,\omega)\dto (M,L)\times (M,-L),
\end{equation}
i.e.\ $\sigma\times \tau$ is a forward Dirac map and $\ker(\sigma\times \tau)\cap \ker(\omega)=0$. Here $G_V$ is endowed with the graph of $\omega$, which is $(\sigma\times \tau)^*(H,-H)$-twisted, and $M\times M$ with the product Dirac structure $L\times -L$, where $-L:=\{v-\alpha :  v+\alpha\in L\}$ is the opposite Dirac structure, which is $(H,-H)$-twisted.

For $H=0$, we have that $\omega=\int_{0}^1(\phi_V^t)^*\br^*\omega_0\ dt$. In this case, the fact that (\ref{eq:presymplectic_realization}) is a presymplectic realization was proven also in \cite{IP}, with a direct Dirac-geometric argument.
%

\subsection{Local integration of multiplicative distributions}\label{section:mult}

We start by recalling the local groupoid version of multiplicative distributions; as a reference to multiplicative distributions on Lie groupoids, see \cite{Maria}.

By a \textbf{local distribution} on a local Lie groupoid $G\rightrightarrows M$ we mean a vector subbundle $H\subset TG$ supported on a neighborhood $U\subset G$ of $M$. A local distribution $H\subset TG$ is called \textbf{multiplicative} if it defines a local Lie subgroupoid over $TM$ of the tangent groupoid $TG\rightrightarrows TM$, i.e.\ $TM\subset H$ and there are open neighborhoods $U_1\subset G$ and $U_2\subset G_{\sigma}\times_{\tau}G$ of $M$ such that \[d\iota(H|_{U_1})\subset H\ \ \ \textrm{and}\ \  d\mu\big((H_{d\sigma}\times_{d\tau}H)|_{U_2}\big)\subset H.\]

Any pointwise surjective $E$-valued local multiplicative 1-form $\omega\in \Omega^1_{\mathrm{loc}}(G;\sigma^*(E))$ ($E$ any representation of $G$) gives rise to a multiplicative distribution: $H:=\ker (\omega)$. 
In fact, every local multiplicative distribution is associated with a surjective multiplicative one-form as above, and this gives the recipe to construct the infinitesimal counterpart of multiplicative distributions. To see this, let $H\subset TG$ be a local multiplicative distribution, and denote
\[H^{\sigma}:=T^{\sigma}G\cap H\ \ \ \textrm{and}\ \ \ \g:=H^{\sigma}|_{M}\subset A.\]
There is a canonical adjoint-like local action of $G$ on the vector bundle $A/\g$, defined as follows:
\[g\cdot[v]:=\big[\theta_G (d\mu(w,v))\big]\in A_{\tau(g)}/\g_{\tau(g)},\ \ \ [v]\in A_{\sigma(g)}/\g_{\sigma(g)},\]
where $w\in H_g$ is any element satisfying $d\sigma(w)=d\tau(v)$ and $\theta_G$ is the Maurer-Cartan form of $G$. The existence of $w$ is insured by the fact that $d\sigma:H\to TM$ is onto, at least in a neighborhood of $M$, and that the action is well-defined is a consequence of the fact that $H$ is a subgroupoid. The associated local multiplicative $1$-form $\omega\in \Omega^1(G;\sigma^*(A/\g))$ is given by:
\[\omega_g(v)=g^{-1}\cdot \big[\theta_G(v-u)\big]\in A_{\sigma(g)}/\g_{\sigma(g)},\ \ \ v\in T_gG,\]
where $u\in H_g$ is any element such that $d\sigma(v)=d\sigma(u)$. Again, using that $H$ is a subgroupoid, one can prove that, in a small enough neighborhood of $M$, $\omega$ is well-defined, that it is multiplicative, and that $\ker(\omega)=H$.

The infinitesimal counterpart of multiplicative distributions are the following objects on a Lie algebroid $A$: a subbundle $\g\subset A$ and a Spencer operator of the form
\[D:\Gamma(A)\rmap \Omega^1(M;A/\g), \ \ \ \ l=\mathrm{pr}:A\rmap  A/\g.\]
The action of $A$ on $A/\g$ is recovered from the relations \eqref{eq:compatibility} as the flat $A$-connection: 
\[\nabla:\Gamma(A)\times\Gamma(A/\g)\rmap\Gamma(A/\g),\ \ \nabla_a(\mathrm{pr}(b))=\iota_{\rho(b)}D(a)+\mathrm{pr}([ a,b]).\]

The differentiation procedure is as follows. Given a local multiplicative distribution $H\subset TG$, the corresponding Spencer operator is explicitly given by $\g=(H\cap T^\sigma G)|_M=H\cap A$, and
\begin{equation}\label{eq:mult}
\iota_vD(a)=[\widetilde{v},a^R] \mod \g,
\end{equation}
for any $a\in\Gamma(A)$, $v\in\X(M)$, where $a^R\in\Gamma(T^\sigma G)$ is the right invariant vector field induced by $a$, and $\widetilde{v}\in\Gamma(H)$ is any extension of $v$ in $H$ around $M$ satisfying that $d\sigma(\widetilde v)=v$. In this case, we say that $H$ integrates $D$.

\vspace*{0.5cm}

For a fixed subbundle $\g\subset A$ and a fixed local Lie groupoid $G\rr M$, the above constructions give one-to-one correspondences between:
\begin{enumerate}
\item germs around $M$ of multiplicative distributions $H\subset TG$ satisfying $H\cap A=\g$;
\item pairs consisting of a germ of a local representation of $G$ on $A/\g$ and a germ of a local multiplicative 1-form $\omega\in \Omega^1_{\mathrm{loc}}(G;\sigma^*(A/\g))$ satisfying $\omega|_{A}=\mathrm{pr}:A\to A/\g$;
\item pairs consisting of a representation of $A$ on $A/\g$ and a degree one Spencer operator of the form $(\mathrm{pr},D)$ on $A$ with values in $A/\g$.
\end{enumerate}

For a spray groupoid $G_V\rr M$, Theorem \ref{theorem:Spencer} makes the integration procedure, i.e.\ the passage from item 3 to item 2 above, also explicit. Namely, given a representation of $A$ on $A/\g$ and a compatible Spencer operator $(\mathrm{pr},D)$ as above, then the corresponding local multiplicative 1-form is given by
\[\omega(v)=\int_0^1T_{\phi_{V}^{\bullet}(a)}^{0,t}\cdot \Lambda_{(\mathrm{pr},D)}\big((\phi_V^t)_*v\big),\ \  \ v\in T_{a}G_V,\]
where $\Lambda_{(\mathrm{pr},D)}\in \Omega^{1}_{\mathrm{lin}}(A;A/\g)$ is the linear form associated to $(\mathrm{pr},D)$ (Lemma \ref{lem:one_to_one}), and $T_{\phi_{V}^{\bullet}(a)}^{0,t}$ is the parallel transport with respect of $\nabla$ from $t$ to $0$ along the path $\phi_{V}^{\bullet}(a)$ (Subsection \ref{subsection:int_spray}).

\subsection{Local integration of Jacobi structures}

Another type of geometric structures which can be encoded by Spencer operators are Jacobi structures. These were first considered independently by Kirillov \cite{Kirillov} and Lichnerowicz \cite{Lichnerowicz}; here we follow the approach to Jacobi manifolds in terms Spencer operators developed in \cite{CS}.

A Jacobi manifold $(M,L,\{,\})$ consists of a line bundle $L$ over the manifold $M$ together with a Lie bracket $\{,\}$ on the space of sections of $L$ which is local:
\[\text{supp}\{u,v\}\subset \text{supp}(u)\cap \text{supp}(v), \ \ \ u,v\in \Gamma(L).\]

The Jacobi bracket induces a Lie algebroid structure on the first jet bundle of $L$
\[q:J^1L \rmap M.\]
The structure maps of this Lie algebroid are uniquely defined by the relations
\begin{equation}\label{eq:bra}\{u,fv\}=f\{u,v\}+\mathcal{L}_{\rho(j^1u)}(f)v, \ \ [j^1u,j^1v]=j^1\{u,v\}, \ f\in C^\infty(M), u,v\in\Gamma(L).\end{equation}
Further, the Lie algebroid $J^1L$ carries a degree one Spencer operator \[D:\Gamma(J^1L)\rmap\Omega^1(M;L),\ \ l=\mathrm{pr}:J^1L\rmap L,\]
which is uniquely determined by the condition $D(j^1u)=0$, $u\in\Gamma(L)$, and the Leibniz identity with respect to $\mathrm{pr}$. The linear one-form associated to $(\mathrm{pr},D)$ is the canonical contact form $\Lambda\in\Omega^1_{\mathrm{lin}}(J^1L;L)$
\[\Lambda_{j^1_xu}=d_{x}(\mathrm{pr}-u\circ q):T_{j^1_xu}(J^1L)\rmap L_x,\]
where we have used the canonical identification $L_x\cong T_{u(x)}(L_x)$.

Finally, the underlying representation $\nabla:\Gamma(J^1L)\times \Gamma(L)\to\Gamma(L)$ is defined by
\begin{equation}
\label{eq:na}\nabla_{j^1u}(v)=\{u,v\}.
\end{equation}


\begin{example}\textbf{(Contact structures as Jacobi brackets)}\label{ex:contact}
Just as symplectic structures correspond to non-degenerate Poisson structures, contact structures correspond to non-degenerate Jacobi structures. Let $H$ be a contact structure on $M$, i.e.\ $H\subset TM$ is a codimension-one distribution which is maximally non-integrable, in the sense that the pairing
\begin{equation}\label{eq:pairing}
(\cdot,\cdot): H_x\times H_x \rmap L_x:=T_xM/H_x,\ \ \  (X,Y)\mapsto [\widetilde{X},\widetilde{Y}]_x  \mod  H_x
\end{equation}
is non-degenerate, where $\widetilde{X}$ and $\widetilde{Y}\in \Gamma(H)$ denote extensions of $X$ and $Y$, resp. The Lie algebra of \textbf{Reeb vector fields} of the contact structure $H$, denoted $\X^{\text{Reeb}}(M,H)\subset \X(M)$, consists of vector fields $R\in \X(M)$ satisfying $[R,\Gamma(H)]\subset \Gamma(H)$. The projection $TM\to L$ induces a linear isomorphism $\X^{\text{Reeb}}(M,H)\cong \Gamma(L)$. This endows $L$ with a canonical Jacobi bracket $\{,\}_H$.
\end{example}

A \textbf{Jacobi map} between Jacobi manifolds $(M_1,L_1,\{\cdot,\cdot\}_{L_1})$ and $(M_2,L_2,\{\cdot,\cdot\}_{L_2})$ consists of a smooth map $\psi:M_1\to M_2$, and a line bundle isomorphism $\Psi:\psi^*L_2\diffto L_1$ such that the map induced on sections $\Gamma(L_2)\to\Gamma (L_1)$, $u\mapsto \Psi(u\circ \psi)$ is a Lie algebra map. Following the previous example, a \textbf{contact realization} of a Jacobi manifold $(M,L,\{,\})$ is a contact manifold $(\Sigma,H)$ and a Jacobi map
\[(\psi,\Psi):(\Sigma,L_H:=T\Sigma/H, \{,\}_H) \rmap (M,L,\{,\})\]
such that $\psi$ is a surjective submersion.

A \textbf{local contact groupoid} consists of a local Lie groupoid $G\rr M$ together with a local multiplicative distribution $H\subset TG$ which is contact. Let $(G,H)$ be a local contact groupoid with Lie algebroid $A$. As discussed in the previous subsection,
$G$ has a canonical local action on the line bundle $L:=A/(H\cap A)$ and a canonical multiplicative 1-form $\omega:TG\to L$, which satisfies $H=\ker \omega$, and so, it induces a fibre-wise invertible line bundle map
\[\overline{\omega}:L_H:=TG/H\diffto \sigma^*L.\]

By adapting the arguments from \cite{CS} to this local setting, one can prove that:
\begin{itemize}
\item There is a unique Jacobi structure $(L,\{\cdot,\cdot\})$ such that $(\sigma,\overline{\omega})$ is a contact realization. This holds because: the Jacobi bracket corresponding to $H$ comes from the isomorphism $\Gamma(L_H)\cong \X^{\text{Reeb}}(G,H)$; and we have the following equality around $M$
\[\omega\big(\X_{\text{L-inv}}^{\text{Reeb}}(G,H)\big)=\sigma^*(\Gamma(L))\subset \Gamma(\sigma^*(L)),\]
where $\X_{\text{L-inv}}^{\text{Reeb}}(G,H)$ denotes the space of local Reeb vector fields on $G$ that are left-invariant (and so, tangent to the $\tau$-fibres).
\item There is a canonical isomorphism of Lie algebroids $\Phi:J^1L\diffto A$, between the Lie algebroid $J^1L$ corresponding to the Jacobi structure on $L$ and the Lie algebroid $A$ of $G$, which is uniquely determined by the relation:
\[\Phi(j^1u)=d\iota(a^L_u)|_{M}\in \Gamma(A), \ \ u\in \Gamma(L),\]
where $a_u^L\in \X_{\text{L-inv}}^{\text{Reeb}}(G,H)$ denotes the left-invariant Reeb vector field satisfying $\omega(a_u^L)=u\circ \sigma$, and $\iota$ denotes the inversion of $G$.
\item Under the isomorphism $\Phi:J^1L\diffto A$, the multiplicative form $\omega$ differentiates to the canonical Spencer operator $(\mathrm{pr},D)$ on $J^1L$.
\end{itemize}

Next, let us explain what our spray method gives in the case of Jacobi manifolds. Consider a Jacobi structure $(M,L,\{\cdot,\cdot\})$. Let $V$ be a spray on the associated Lie algebroid $J^1L$, and let $G_V\rr M$ be the corresponding local spray groupoid. By the discussion in the previous section, the multiplicative $L$-valued form integrating the Spencer operator $(\mathrm{pr},D)$ is given by:
\[\omega\in \Omega^{1}(G_V,\sigma^*(L)), \ \ \ \omega_a=\int_{0}^1 (ta)^{-1}\cdot(\phi_V^t)^*\Lambda_{(\mathrm{pr}, D)}\ dt.\]
Moreover, the corresponding local multiplicative distribution
$H:=\ker(\omega)\subset TG_V$
is contact. For this note that
\begin{itemize}
 \item $\ker(\mathrm{pr})\subset J^1L$ can be canonically identified with $T^*M\otimes L$;
 \item along the unit section, we have that $H|_M=TM\oplus \ker(\mathrm{pr})\cong TM\oplus T^*M\otimes L$; and, under this isomorphism, for $a\in \Omega^1(M,L)$ we have that $D(a)=-a$;
\item along the unit section, by (\ref{eq:mult}), the pairing (\ref{eq:pairing}) satisfies:
  \[(X,a)=\mathrm{pr}([\widetilde{X},a^R])_x=i_{X}D(a)=-i_Xa\in \Gamma(L),\]
  for all $X\in \X(M)$, $a\in \Omega^1(M,L)$, where $\widetilde{X}\in\Gamma(H)$ is an extension of $X$ such that $d\sigma(\widetilde X)=X$, and $a^R\in\Gamma(T^\sigma G)$ is the right invariant vector field induced by $a$.
 \end{itemize}
The last item implies that the pairing is non-degenerate along $M$; hence it is so in a neighborhood of $M$, which can be taken to be the domain of $H$. So $(G_V,H)\rr M$ is a local contact groupoid. Finally, by the discussion above, we have that $(G_V,H)\rr M$ integrates the Jacobi manifold $(M,L,\{,\})$, in the sense that the the source map $\sigma:G_V\to M$ together with $\overline{\omega}:TG_V/H \to\sigma^* L$ define a contact realization (the proof of \cite[Theorem 2]{CS} applies directly in this local case).

\begin{example}
Assume that $L$ is trivializable 
and fix an isomorphism $L\cong \R\times M$. Then there exist a bivector field $\pi\in \X^2(M)$ and a vector field $R\in \X(M)$ such that \cite{Kirillov}
\[\{u,v\}=\pi(du,dv)+\langle R , du \rangle v-u\langle R , dv \rangle, \ \ u,v\in \Gamma(L)\cong \C(M).\]
The Jacobi identity for $\{\cdot,\cdot\}$ is equivalent to the equations:
$[\pi,\pi]=2R\wedge\pi,  \  [\pi,R]=0.$
In this case, we have the identification
\[J^1L\cong  \R\times T^*M, \ \ j^1_xu\mapsto (u(x),d_xu).\]
The canonical Spencer operator acts as follows:
\[D:\C(M)\oplus \Omega^1(M)\to \Omega^1(M), \ \ D(u,a)=du-a.\]
The corresponding linear 1-form is the contact form:
\[\Lambda=d\mathrm{pr}_1-\mathrm{pr}_2^*\alpha_0\in \Omega^1(\R\times T^*M),\]
where $\alpha_0\in \Omega^1(T^*M)$ is the tautological 1-form.

For the Lie algebroid structure on $J^1L\cong \R\times T^*M$, see e.g.\ \cite{CrainicZhu}; let us just recall the anchor:
\[\rho(u,a)=\pi^{\sharp}(a)-u R,\]
and the action of $J^1L\cong \R\times T^*M$ on $L\cong \R\times M$ \[\nabla_{(u,a)}v=\rho(u,a)(v)+\langle R,a\rangle v,\]
i.e.\ the action corresponding to the Lie algebroid cocycle
\[\widetilde{R}:J^1L\cong \R\times T^*M\to \R, \ \ \ (u,a)\mapsto \langle R_{q(a)},a \rangle.\]

Let $V$ be a spray for $J^1L\cong \R\times T^*M$, with spray groupoid $G_V$.
By Subsection \ref{subsection:cocyles}, the above cocycle integrates to the groupoid cocycle
\[G_V\dto \R, \ \ (u,a)\mapsto\int_0^1 \langle R,a_{t}\rangle\ dt=\int_0^1\widetilde{R}(\phi_V^t)\ dt,\]
where $a_t$ is the second component of the flow of $V$, i.e.\ $\phi_V^t(u,a)=(u_t,a_t)$. Therefore, the action of $G_V$ on $L\cong \R\times M$ is:
\[(u,a)\cdot (v,x)=\big(e^{\int_0^1\widetilde{R}(\phi_V^t) dt}v, \tau(u,a)\big)=\big(e^{\int_0^1 \langle R,a_{t}\rangle dt}v, q(a_1)\big),\ \ a\in T^*_xM, \ u,v\in \R.\]
Since $\phi^s_V(tu,ta)=t\phi_V^{st}(u,a)$, note that $(tu,ta)\cdot (v,x)=(e^{\int_0^t\langle R,a_{s}\rangle ds}v,\tau(tu,ta))$. Thus, the multiplicative 1-form integrating $(\mathrm{pr},D)$ becomes:
\[\omega_{(u,a)}=\int_0^1e^{-\int_0^t\langle R, a_s\rangle ds }(\phi_V^t)^*(d\mathrm{pr}_1-\mathrm{pr}_2^*\alpha_0)\ dt.\]
\end{example}


\begin{thebibliography}{9}
\bibitem{CamiloMarius} C.~Arias~Abad, M.~Crainic, The Weil algebra and the Van Est isomorphism, \emph{Ann. Inst. Fourier (Grenoble)} \textbf{61} (2011), no.\ 3, 927--970.
\bibitem{BaGu} M.~Bailey, M.~Gualtieri, Integration of generalized complex structures, preprint (2016) arXiv:1611.03850.
\bibitem{BX} D.~Broka, P.~Xu, Symplectic realizations of holomorphic Poisson manifolds, preprint (2015) arXiv:1512.08847.
\bibitem{BC} H.~Bursztyn, A.~Cabrera, Multiplicative forms at the infinitesimal level, \emph{Math. Ann.} Vol. 353, (2012) no.\ 5, 663--705.
\bibitem{BCWZ} H.~Bursztyn, M.~Crainic, A.~Weinstein, C.~Zhu, Integration of
twisted Dirac brackets, {\em Duke Math. J.} {\bf 123} (2004), 549--607.
\bibitem{BDtensor} H.~Bursztyn, T.~Drummond, Lie theory of multiplicative tensors, preprint.
\bibitem{CDr} A.~Cabrera, T.~Drummond, Van est isomorphism for homogeneous cochains, \emph{Pac. J. Math.}  287-2 (2017), 297--336. 
\bibitem{love} A.~Cabrera, I.~M\u{a}rcu\cb{t}, M.A.~Salazar, On local integration of Lie brackets, to appear in \emph{J. Reine Angew. Math.}
\bibitem{CDW} A.~Coste, P.~Dazord, A.~Weinstein, Groupo\"ides symplectiques, \emph{Publications du D\'epartement de Math\'ematiques. Nouvelle S\'erie. A,} Vol. 2, iii, 1--62, Publ. D\'ep. Math. Nouvelle S\'er. A, 87-2, Univ. Claude-Bernard, Lyon, 1987.
\bibitem{CrGCS} M.~Crainic, Generalized complex structures and Lie brackets, \emph{Bull. Braz. Math. Soc. (N.S.)}, \textbf{42} (2011), no. 4, 559--578.
\bibitem{CF1} M.~Crainic, R.L.~Fernandes, Integrability of Lie brackets, \emph{Ann. of Math. (2)} \textbf{157} (2003), no. 2, 575--620.
\bibitem{CF2} M.~Crainic, R.L.~Fernandes, Lectures on integrability of Lie brackets, \emph{Geom. Topol. Monogr.}, \textbf{17}, (2011), 1--107.
\bibitem{CM} M.~Crainic, I.~M\u{a}rcu\cb{t}, On the existence of symplectic realizations, \emph{J.\ Symplectic Geom.}, \textbf{9}, (2011), no. 4, 435--444.
\bibitem{Maria} M.~Crainic, M.A.~Salazar, I.~Struchiner, Multiplicative forms and Spencer operators, \emph{I. Math. Z.}, \textbf{279}, (2015), 939-979.
\bibitem{CS} M.~Crainic, M.A.~Salazar, Jacobi structures and Spencer operators, \emph{J.\ Math.\ Pures Appl. (9)}, \textbf{103}, (2015), no.\ 2, 504--521.
\bibitem{CrainicZhu} M.\ Crainic, C.\ Zhu, Integrability of Jacobi and Poisson structures, \emph{Ann.\ Inst.\ Fourier (Grenoble)} \textbf{57} (2007) 1181--1216.
\bibitem{DrE} T.~Drummond, L.~Egea, Differential forms with values in VB-groupoids, Preprint 	arXiv:1804.05289.
\bibitem{Fe1} R.L.~Fernandes, Lie algebroids, holonomy and characteristic classes, \emph{Adv.\ Math.}\ \textbf{170} (2002), no. 1, 119--179.
\bibitem{IP} P.~Frejlich, I.~M\u{a}rcu\cb{t}, On dual pairs in Dirac geometry, \emph{Math. Z.},  \textbf{289} (2018), no. 1-2, 171-200.
\bibitem{Raj} A.~Gracia-Saz, R.~Mehta, Lie algebroid structures on double vector bundles and representation theory of Lie algebroids, \emph{Adv. Math.} \textbf{223} (2010), 1236--1275.
\bibitem{Gualtieri} M.~Gualtieri, Generalized complex geometry, \emph{Ann. of Math. (2)} \textbf{174} (2011), no.\ 1, 75--123.
\bibitem{ILX} D.~Iglesias Ponte, C.~Laurent-Gangoux, P.~ Xu, Universal lifting theorem and quasi-Poisson groupoids, \emph{J.\ Eur.\ Math.\ Soc.\ (JEMS)}, \textbf{14}, no. 3, (2012), 681--731.
\bibitem{Kirillov} A.A.~Kirillov, Local Lie algebras, \emph{Usp. Mat. Nauk} \textbf{31} (1976) 57--76.
\bibitem{Yvette} Y.~Kosmann-Schwarzbach,  F.~Magri, Poisson-Nijnehuis structures, \emph{Ann. Inst. Henri Poincar\'e} \textbf{53} (1990), 35--81.
\bibitem{L-BM} D.~Li-Bland, E.~Meinrenken, On the van Est homomorphism for Lie groupoids, \emph{Enseign. Math.} \textbf{61} (2015), no. 1--2, 93--137.    
\bibitem{Lichnerowicz} A.~Lichnerowicz, Les vari\'et\'es de Jacobi et leurs alg\`ebres de Lie associ\'es, \emph{J. Math. Pures Appl.} (9) \textbf{57} (1978) 453--488.
 \bibitem{Mac} K.~Mackenzie, \emph{General theory of Lie groupoids and Lie algebroids}, Cambridge University Press, (2005).
\bibitem{Peta} F.~Petalidou, On the symplectic realization of Poisson-Nijenhuis manifolds, Preprint arXiv:1501.07830.
\bibitem{Mtesis} M.A.~Salazar, \emph{Pfaffian groupoids}, Ph.D. thesis, Utrecht University, 2013.
\bibitem{StiXu} M.~Sti\'enon, P.~Xu, Poisson quasi-Nijenhuis manifolds, \emph{Comm. Math. Phys.} \textbf{270} (2007), no.\ 3, 709--725.
\bibitem{Ori} O.~Yudilevich, The role of the Jacobi identity in solving the Maurer-Cartan structure equation, \emph{Pacific J.\ Math.}\ \textbf{282} (2016), no.\ 2, 487--510.
\end{thebibliography}
\end{document}